\documentclass[final]{siamltex}
\usepackage{mathrsfs}
\usepackage{amsmath,amssymb}
\usepackage{color} 
\usepackage{epsfig}
\usepackage{graphicx}
\usepackage[english]{babel}
\usepackage{algorithm,algorithmic}
\usepackage[]{showkeys}
\usepackage{makros1}

\title{Multilevel ensemble Kalman filtering}

\author{H{\aa}kon Hoel\thanks{Department of Mathematics, University of Oslo, Norway
  (haakonah@math.uio.no), and Applied Mathematics and Computational Sciences, KAUST,
  Thuwal, Saudi Arabia}
\and
Kody J. H. Law\thanks{Computer Science and Mathematics Division, Oak Ridge National Laboratory, Oak Ridge, TN, USA (lawkj@ornl.gov)} 
\and
Raul Tempone\thanks{Applied Mathematics and Computational Sciences, KAUST, Thuwal, Saudi Arabia (raul.tempone@kaust.edu.sa)}
}

\begin{document}

\maketitle

\begin{abstract}
  \kl{
  This work embeds a multilevel Monte Carlo (MLMC) sampling strategy
  into the Monte Carlo step of the ensemble Kalman filter (EnKF) in
  the setting of finite dimensional 
  signal evolution and 
  noisy discrete-time observations.  The
  signal dynamics is assumed to be governed by a
  stochastic differential equation (SDE), and a hierarchy of time grids is
  introduced for multilevel numerical integration of that 
  SDE. The resulting multilevel ensemble Kalman filter method (MLEnKF)
  is proved to asymptotically outperform EnKF in terms of
  computational cost vs.~approximation accuracy. The theoretical
  results are illustrated numerically.
 }
 
   \bigskip
   \noindent \textbf{Key words}: Monte Carlo, multilevel, filtering, Kalman filter, ensemble Kalman filter.
   
   \noindent \textbf{AMS subject classification}:
   65C30, 65Y20. 
\end{abstract}


\pagestyle{myheadings}
\thispagestyle{plain}
\markboth{Multilevel ensemble Kalman filtering}{Multilevel ensemble Kalman filtering}




\section{Introduction}
\label{sec:intro}

Filtering refers to the sequential estimation of the state $u$ and/or parameters $p$ of a system through sequential 
incorporation of online data $y$. The most complete estimation of the state $u_n$ at time $n$ is given by its 
probability density conditional on the observations up to the given time $\pi(u_n|y_1,\ldots, y_n)$ \cite{jaz70, BC09}.   
For linear Gaussian systems the analytical solution may be given in closed form, 
via an update formulae for the mean and covariance known as the Kalman filter \cite{kalman1960new}.
However, in general there is no closed form solution.  One must therefore resort to either algorithms which approximate 
the probabilistic solution by leveraging ideas from control theory 
\cite{kal03,jaz70}, or Monte Carlo methods to approximate the filtering distribution itself 
\cite{BC09, doucet2000sequential, del2004feynman}.  The ensemble Kalman filter (EnKF) 
\cite{burgers1998analysis, evensen2003ensemble} combines elements of both approaches.
In the linear Gaussian case, it converges to the Kalman filter solution \cite{mandel2011convergence}, and even 
in the nonlinear case, under suitable assumptions it converges \cite{le2011large, law2014deterministic} to what one may argue is the optimal filter
among those which incorporate the data linearly \cite{law2014deterministic, luenberger1968optimization, pajonk2012deterministic}.
\kl{ In the case of spatial models approximated on a numerical grid, }
the state space itself may become very high-dimensional and even the linear solves may become intractable.
Therefore, one may be inclined to use the EnKF filter even for linear Gaussian problems in which the solution is
intractable despite being given in closed form on paper by the Kalman filter.

For problems which admit hierarchies of approximations with cost inversely proportional to accuracy, 
it is natural to leverage solutions to less expensive and less accurate approximations in order to accelerate 
the convergence of the more expensive and more accurate approximations.  
This idea originated in the iterative solution of numerical approximation of elliptic PDE as early as \cite{fedorenko1961relaxation}.
In the context of iterative solution of PDE, the methodology, which may be used both for solution as well
as pre-conditioner, has become known as multigrid -- see \cite{hackbusch1985multi} for a general reference.  
The same idea may be applied in the context of Monte Carlo approximation of random fields 
as proposed in \cite{heinrich2001multilevel}, and later studied in detail in the context of SDE in \cite{Giles08} 
and PDE in \cite{Cliffe11}.  There has been an explosion of recent activity since \cite{Giles08}, 
including for example {~\cite{Giles14,CollierBIT,Hoel14,Chernov14,Mishra12}}.  Even more recently, work is beginning to emerge extending the same 
multilevel framework beyond "vanilla" Monte Carlo to its manifestation in the context of Bayesian inference, 
anywhere that one has a discretization error inversely proportional to cost in the computation of a single 
sample and the Monte Carlo $\cO(M^{-1/2})$ rate of convergence.  Examples in the 
context of Markov Chain Monte Carlo appeared in \cite{ketelsen2013hierarchical, hoang2013complexity}.
To the knowledge of the authors there has yet to be extension of the methodology to the filtering context.  
As a first step, this work explores the extension of the EnKF 
to its multilevel implementation, which is naturally referred to as multilevel EnKF (MLEnKF).  In the case of linear
Gaussian problems, the limiting distribution is the gold-standard Bayesian posterior distribution, while in the 
non-Gaussian case it is something else (see, e.g. \cite{law2014deterministic, luenberger1968optimization, pajonk2012deterministic} and references therein for further discussion).

The rest of the paper will be organized as follows.  In section \ref{sec:kalman} the filtering problem will be introduced,
the Kalman filter and EnKF will be reviewed, 
and the new multilevel EnKF (MLEnKF) will be introduced for the first time in sub-section \ref{subsec:mlenkf}.  
In section \ref{sec:theory} it is proven that indeed the MLEnKF inherits almost the same favorable asymptotic "cost-to-$\varepsilon$" as 
the standard MLMC for a finite time horizon, and its mean-field limiting distribution is the filtering distribution in the linear and
Gaussian case. 
In section \ref{sec:numerics} the theory is illustrated with numerical examples
of the Ornstein-Uhlenbeck stochastic differential equation, and the geometric Brownian motion.
Finally, conclusions and future directions are presented in section
\ref{sec:conclusion}.




\section{Kalman filtering}
\label{sec:kalman}

Here the filtering problem will be introduced in section \ref{subsec:genFiltering}, and then the Kalman filter will be introduced
for linear Gaussian state-space models in section \ref{subsec:kalmanFiltering}.
In section \ref{subsec:enkf}, its Monte Carlo implementation of the Kalman filter will be introduced, 
which is known as the ensemble Kalman filter (ENKF).
In section \ref{subsec:mlenkf}, a the multilevel implementation is introduced for the first time.

\subsection{General set-up}
\label{subsec:genFiltering}

Let $(\Omega,\mathcal{E}, \bbP)$ be a complete probability space, where 
$\Omega$ is the set \hh{of all possible outcomes}, $\mathcal{E}$ is the sigma algebra generated by $\Omega$
and $\bbP$ is the associated probability measure.
Consider the general stochastic signal evolution for the random variables 
$u_n :\Omega \rightarrow \bbR^d$, with $d<\infty$, 
\begin{equation}\label{eq:psiDefinition}
u_{n+1} = \Psi 
(u_n), 
\end{equation}
for $n=0,1,\ldots,N-1$.
Given the history of a noisy signal observation
\[
y_n = H u_n + \eta_n, 
\]
where $H\in\bbR^{m\times d}$ and $\eta_n$ are i.i.d. with $\eta_1 \sim
N(0,\Gamma), \Gamma \in \bbR^{m\times m}$ symmetric positive definite,
\kl{the objective is to track the signal $u_n$ given observations $Y_n =
Y_n^{\rm obs}$ with $Y_n = (y_1, y_2, \ldots, y_n) $ and $Y_n^{\rm
  obs} = (y_1^{\rm obs}, y_2^{\rm obs}, \ldots, y_n^{\rm obs})$ the given observed 
  realization of $Y_n$.  
  In other words, the aim is to approximate the random variable $u_n|(Y_n = Y_n^{\rm obs})$.  
  Notice that under the given assumptions this is a hidden Markov model 
  and the density of the random variable we
seek to approximate admits the following sequential structure
\begin{align}
\label{eq:filteringdist}
\pi_{u_n}(u | Y_n=Y_n^{\rm obs}) & = \frac{\pi_{y_n}(y_n^{obs}|u_n = u) \pi_{u_n}(u |Y_{n-1} = Y_{n-1}^{\rm obs})}{\pi_{y_n}(y_n^{\rm obs}|Y_{n-1} = Y_{n-1}^{\rm obs})}, \\
\nonumber
\pi_{u_n}(u | Y_{n-1}=Y_{n-1}^{\rm obs}) & = \int_{\bbR^d} \pi_{u_n}(u |u_{n-1}=v)\pi_{u_{n-1}}(v|Y_{n-1} = Y_{n-1}^{\rm obs}) \, \rD v , \\
\nonumber
\pi_{y_n}(y_n^{obs} |Y_{n-1} = Y_{n-1}^{\rm obs}) & = \int_{\bbR^d} \pi_{y_n}(y_n^{\rm obs}|u_n=u) \pi_{u_n}(u |Y_{n-1} = Y_{n-1}^{\rm obs}) \, \rD u,
\end{align}
Here, $\pi_{X_1}(x_1|X_2 = x_2)$ denotes the marginal probability density 
of $X_1$ conditional that $X_2 = x_2$.}

It will be assumed that $\Psi(\cdot)$ cannot be evaluated exactly, but rather only approximately,
and that there exists a hierarchy of accuracies at which it can be evaluated each with its associated cost.
The explicit dependence on $\omega$ will be suppressed where confusion is not possible.  
In particular, we will be concerned herein with the case in which 
\kl{$u_{n+1}= \Psi 
(u_n) := u(1;u(0)=u_n)$ is given by the 
evolution at $t=1$ of the following SDE
\begin{equation}\label{eq:sde}
\begin{split}
du(t) & = a(u(t)) dt + b(u(t) ) dW(t+n),
\qquad  t \in (0, 1), 
\end{split}
\end{equation}
given the prescribed initial condition $u(0)=u_n$, 
where $a:\bbR^d 
\rightarrow \bbR^d$, 
$b: \bbR^d 
\rightarrow \bbR^{d \times r}$}, and $W:\Omega \times [0, \infty) \rightarrow \bbR^r$
is an $r-$dimensional 
Wiener process.
Furthermore, $a$ and $b$ will satisfy the following 
conditions
\begin{equation}\label{eq:lipcoeffs}
|a(u) - a(v)| + |b(u) - b(v)| \leq c |u-v| ~{\rm for~all}~ u,v \in
\bbR^{d}~{\rm and~some~} c>0. 
\end{equation}
This fits into the framework of \eqref{eq:psiDefinition}.  
Notice that the randomness comes from the initial condition and the Wiener
process $W$ only, and not the coefficients, however it could in principle come also from the latter.
Indeed if the analytical solution is unknown, then the system above must be approximated, leading to the 
hierarchy of approximations to $\Psi
(\cdot)$.  
In particular, 
denote by $\{\psiL \}_{\ell=0}^\infty$ a hierarchy of approximations to the solution 
$\Psi:=\Psi^{\infty}$ of~\eqref{eq:sde}.  First some assumptions must be made.
 \begin{assumption}  For every $p \geq 2$, 
 the solution operators $\{\psiL\}_{\ell=0}^\infty$ satisfy the following conditions, for
 some $0<c_\Psi<\infty$ depending on $\Psi$:
\begin{itemize} 
\item[(i)] $\|\psiL(u) -\psiL(v) \|_p < c_{\Psi} \|u-v\|_p $, 
\item[(ii)] $\|\psiL(u)\|_p^p \leq c_{\Psi} (1+\|u\|_p^p)$,
\end{itemize}
where 
the following notation is introduced $\|u\|_p := \E{\abs{u}^p  }^{1/p}$.
\label{ass:psilip}
\end{assumption}

For many numerical solvers the assumptions can be verified by application 
of Gronwall's inequality, cf.~\cite{GrahamTalay,Carlsson10}.
\kl{For notational simplicity, we 
consider the particular case in which 
\eqref{eq:sde} is {\it autonomous}, such that the coefficients on the right-hand side do not depend
explicitly on time.  
Note that the results easily extend to the
non-autonomous case, provided the given assumptions
on $\Psi$ are uniform with respect to $\{\Psi_n\}_{n=1}^N$.} 
The specialization is merely for notational convenience.

\subsection{Linear Gaussian case}
\label{subsec:kalmanFiltering}

Consider the linear instance of~\eqref{eq:psiDefinition}, in which
\begin{equation}\label{eq:psilinear}
\Psi(u_n) = A u_n + \xi_n,
\end{equation}
where $A\in \bbR^{d\times d}$, and $\xi_n$ are i.i.d. normal random variables with 
$\xi_1 \sim N( 0 , \Sigma) , \Sigma \in \bbR^{d\times d}$ symmetric positive definite. 
This case arises when the coefficients of \eqref{eq:sde} are given by
\begin{equation}
\label{eq:linnonaut}
a(u) = a_1
u + a_0
\quad {\rm and} \quad b(u)
= b_0.
\end{equation}
Again we suppress the possible time-dependence of the random maps 
$\Psi$ and matrices $A$ and $\Sigma$ just to simplify notation. 
For this class of problems, given a Gaussian initial condition, 
the filtering distribution \eqref{eq:filteringdist} is known to
be Gaussian, and is therefore defined uniquely by its mean and covariance.
\kl{ Kalman filtering provides a two step iterative procedure for computing the 
mean and covariance of $\hhv_{n+1}:=u_{n+1}|(Y_{n+1} = Y_{n+1}^\obs)$ given
$\hhv_{n} := u_n|(Y_n = Y_n^{\obs}) \sim N(\hhm_{n}, \hhc_{n})$ where 
\[
\hhm_{n} := \E{u_n|Y_n = Y_n^\obs } \quad  \text{and} \quad  \hhc_{n} := \E{ (u_{n} - \hhm_{n}) (u_{n} - \hhm_{n})^\transpose | Y_n =Y_n^{\obs}}.  
\]
The classical Kalman filter consists of a two-step formula 
which maps the distribution of $\hhv_{n}$ 
to the distribution of $\hhv_{n+1}$.
In particular, the map from $(\mean{n} , \widehat{C}_{n})$ to $(\mean{n+1} , \widehat{C}_{n+1})$ 
is described by the following two prediction equations and two update equations }
\begin{align}\label{eq:updateKf}
 \nonumber
 \meanHat{n+1} & =  A \mean{n},				 &\covHat{n+1} & =  A\widehat{C}_{n} A^\transpose + \Sigma\\
 \nonumber
 \mean{n+1} & =  (I - K_{n+1} H ) \meanHat{n+1} + K_{n+1} y_{n+1}^\obs,  	& \widehat{C}_{n+1}     & =  (I - K_{n+1} H ) \covHat{n+1}.
\end{align}
where \kl{the \it{Kalman gain}} $ K_{n+1}$ is given by
 \begin{align}
  \nonumber
K_{n+1} &= \covHat{n+1} H^\transpose S_{n+1}^{-1}, & S_{n+1} & = \Gamma + H\covHat{n+1} H^\transpose .
\end{align}

\begin{remark}
When observations are in a lower dimension than the true signal, the iterations are more 
efficiently computed by introducing $d_{n+1} = y_{n+1}^\obs - H \meanHat{n+1}$ and 
updating the mean as follows 
\[
 \mean{n+1}    = \meanHat{n+1} + K_{n+1} d_{n+1}.
\]
It is nonetheless instructive to observe the alternative form presented above, 
as it is the properties of the operators  $I-K_{n+1}H$ which are responsible for stability of the algorithm
\cite{lawstuartDA2015}.
\end{remark}

\subsection{EnKF}
\label{subsec:enkf}

EnKF uses an ensemble of particles to estimate means and covariance
matrices appearing in the Kalman filter, however the framework can be
generalized to non-Gaussian models. Let $\vHat{n,i}$ and 
$\widehat{v}_{n,i}$ 
respectively denote the prediction and update of
the $i$-th particle at simulation time $t_n=n$.  
One EnKF two-step transition consists not of the propagation of a mean and covariance as in the 
original Kalman filter, 
but instead the propagation of an ensemble $\{\widehat{v}_{n,i}\}_{i=1}^M  \mapsto \{\widehat{v}_{n+1,i}\}_{i=1}^M$.
\footnote{Due to the implicit linear and Gaussian assumptions underlying the formulation, 
one may determine that it is reasonable to summarize the ensemble in its sample mean and 
covariance and indeed this is often done.  In this case, one may  construct a Gaussian from 
the empirical statistics and resample from that.}
This procedure consists nonetheless in the prediction and update steps. 
In the prediction step, 
$M$ particle paths are computed over one interval, i.e., 
 \begin{equation}
  \vHat{n+1}(\omega_i) = \Psi(\widehat{v}_{n}(\omega_i),\omega_i) 
   \end{equation}
for $i=1, \ldots, M$,  where $v_{n} (\omega_i) := v_{n,i}$ 
denotes a realization corresponding to the event sample $\omega_i$ of the random variable 
$v_{n} 
: \Omega \rightarrow \bbR^d$, 
and $ \Psi(\cdot,\omega_i)$ signifies the corresponding realization of the map for a given initial condition.
Indeed the notation for random variable realizations, e.g. $\xi_{n,i}$ and $\xi_{n}(\omega_i)$, will be used interchangeably 
where confusion is not possible.  The impetus for introduction of the latter notation will become apparent in the next section.
For this presentation it suffices to assume a single infinite precision map, however there indeed may also be 
numerical approximation errors, i.e. $\Psi^L$ may be used in place of $\Psi$ for some satisfactory
resolution $L$.
\kl{The prediction step is completed by using the particle paths to compute sample mean and covariance:
  \[
  \begin{split}
  \meanHatMC{n+1}  		& = E_M[\vHat{n+1}] \\
  \covHatMC{n+1}   		& = \cov_M[\vHat{n+1}] 
  \end{split}
  \]\
  where the following notations are introduced 
\begin{equation}\label{eq:sampleAvg}
E_M[v] \coloneq \frac{1}{M} \sum_{i=1}^M  v(\omega_{i} ),
\end{equation}
\begin{equation}\label{eq:sampleCov}
\cov_M[u,v] \coloneq 
E_M[u v^\transpose] - E_M[u]  \big(E_M[v]\big)^\transpose,
\end{equation}
as well as the shorthand $\cov_M[u] \coloneq \cov_M[u,u]$.
}
The update step consists of computing (1) auxillary matrices
\[
 S^{\rm MC}_{n+1}            = H\covHatMC{n+1} H^\transpose + \Gamma \text{ and }   K^{\rm MC}_{n+1}   = \covHatMC{n+1} H^\transpose (S^{\rm MC}_{n+1})^{-1}, 
\]
and (2) measurement corrected particle paths for $i=1,2, \ldots, M$,
 \[ 
 \begin{split}
  \yTilde{n+1,i}	 & = y_{n+1}^\obs + \eta_{n+1,i},\\
  \widehat{v}_{n+1,i} 
  & = (I - K^{\rm MC}_{n+1}H) \vHat{n+1,i}
  + K^{\rm MC}_{n+1} \yTilde{n+1,i}, 
\end{split}
\]
where 
$\{ \eta_{n+1,i} \}_{i=1}^M$ are i.i.d.~with $\eta_{n+1,1} \sim N(0, \Gamma)$. 
This last procedure may appear somewhat ad-hoc.  Indeed it was originally introduced 
in \cite{burgers1998analysis} to correct the statistical error induced in its absence in implementations 
following the original formulation of the ensemble Kalman filter in \cite{evensen1994sequential}.
It has become known as the perturbed observation implementation.
Due to the form of the update, all ensemble members are correlated to one another after 
the first update.  So, 
the ensemble is no longer Gaussian after the first update.  
The measurement corrected sample mean and covariance, which need not be computed,
would be given by:
\[
  \begin{split}
  \meanMC{n+1}  		& = E_M[\widehat{v}_{n+1}],\\
  \covMC{n+1}   		& = \cov_M[\widehat{v}_{n+1}]. 
  \end{split}
\]

\kl{The sample empirical distribution is defined by 
\begin{equation}
\mu_n^{\rm MC} = \frac1M \sum_{i=1}^M \delta_{v_n(\omega_{i})},
\label{eq:emp}
\end{equation}
and, for $\varphi:\bbR^d\rightarrow\bbR$, the following shorthand notation is introduced
$\mu_n^{\rm MC} (\varphi) = \int \varphi d\mu_n^{\rm MC} = E_M[\vHat{n}]$.
It was shown in \cite{mandel2011convergence, le2011large}
that if $\Psi$ is of the form \eqref{eq:psilinear} 
and $\bbE|v_0|^p <\infty$ for all $p\geq 2$, then 
for all Lipschitz $\varphi:\bbR^d \rightarrow\bbR$ and all $p\geq 2$,
\begin{equation}
\left(\bbE |\mu_n^{\rm MC}(\varphi) - \mu_n(\varphi) |^p\right)^{1/p} \lesssim M^{-1/2},
\label{eq:empconv}
\end{equation}
where $\mu_n$ is the filtering distribution.  
The notation $f(M) \lesssim g(M)$ here is used to denote $f(M) = \cO(g(M)).$
}

\subsection{Multilevel EnKF}
\label{subsec:mlenkf}

MLEnKF computes particle paths on a hierarchy of accuracy levels, in
this case given by increasing refinement of the temporal
discretization.  
Let $\vHatL{\ell}{n}$, $\vL{\ell}{n}$ respectively denote the
prediction and update of a particle on solution level $\ell$ at
simulation time $t_n$.  
A solution on level $\ell$ is computed by the numerical integrator $\vHatL{\ell}{n+1} = \psiL(\vL{\ell}{n})$.
Furthermore, let the \kl{difference} operator 
for level $\ell$ be given by 
\begin{equation}\label{eq:dlDef}
\DlVHat{n}(\omega)  := 
\begin{cases} 
\vHatL{0}{n} (\omega), & \text{if } \ell =0,\\
\vHatL{\ell}{n}(\omega) - \vHatL{\ell-1}{n} (\omega)  , & \text{else if } \ell >0.
\end{cases}
\end{equation}
Then the transition from approximation of the distribution of {$\hhv_n$ }
to the distribution of {$\hhv_{n+1}$} 
in the MLEnKF framework consists of the predict/update step of generating {\it pairwise coupled} particle realizations 
on a set of levels $\ell=0,1,\ldots, L$.  However, it is important to note that here one has 
correlation between pairs and also between levels due to the update, 
unlike the standard MLMC in which one has i.i.d. pairs.  This point will be very important,
and we return to it in the following section.

Similarly to the standard EnKF, the MLEnKF transition is between {\it multilevel} ensembles \kl{
$\{[\hhv^{\ell}_{n}(\omega_{\ell,i}), \hhv^{\ell-1}_{n}(\omega_{\ell,i})]_{i=1}^{M_\ell}\}_{\ell=0}^L \mapsto 
\{[\hhv^{\ell}_{n+1}(\omega_{\ell,i}), \hhv^{\ell-1}_{n+1}(\omega_{\ell,i})]_{i=1}^{M_\ell}\}_{\ell=0}^L$, 
with the convention that $\hhv^{-1}_{k}:=0$ for all $k$ for ease of notation.}
This consists, as for EnKF, of the predict and update steps. 
In the predict step, particle paths are first computed on a hierarchy of levels. That is, 
the particle paths are computed one step forward by
\begin{equation}\label{eq:DlvHatDef}
\begin{split}
  \vHatL{\ell-1}{n+1}(\omegaLI)  	& = \Psi^{\ell-1}(\vL{\ell-1}{n}(\omegaLI),\omegaLI),\\
  \vHatL{\ell}{n+1}(\omegaLI) 	 	& = \psiL(\vL{\ell}{n}(\omegaLI), \omegaLI),
  \end{split}
\end{equation}
for the levels $\ell =0,1,\ldots,L$ and level particles $i = 1,2, \ldots, M_\ell$
(where for convenience we introduce the convention that $\vHatL{-1}{} :=0$). 
Here the introduction of noise in the second argument of the $\psiL$ are correlated only within 
pairs, and are otherwise independent. 
Thereafter, sample mean and covariance matrices are computed
as a sum of sample moments over all levels:
\[
\begin{split}
\meanHatMLMC{n+1} &= \sum_{\ell=0}^L E_{M_\ell}[\DlVHat{n+1}(\omega_{\ell,\cdot})],\\
\covHatMLMC{n+1}  &= \sum_{\ell=0}^L \cov_{M_\ell}[\vHatL{\ell}{n+1}(\omega_{\ell,\cdot}) ] - 
\cov_{M_\ell}[ \vHatL{\ell-1}{n+1}(\omega_{\ell,\cdot} )],
\end{split}
\]
where we recall the sample moment notation~\eqref{eq:sampleAvg} and~\eqref{eq:sampleCov}.

It is necessary for stability of the algorithm that the sample covariance
appearing in the denominator of the gain is positive semi-definite, a condition
which is {\it not} guaranteed for multilevel estimators.  
This will therefore be {\it imposed} in the algorithm.  It would be of independent interest 
to devise multilevel estimators which preserve positivity without such imposition.  
Let
$$
C^{\rm ML}_n = \sum_{k=1}^d \lambda_k q_k q_k^\transpose
$$
denote the eigenvalue decomposition of the symmetric multilevel covariance.  Notice \kl{
that the condition min$_k(\lambda_k) \geq 0$ may not hold.
Define
\begin{equation}
\tilde{C}^{\rm ML}_n = \sum_{k=1; \lambda_k > 0}^d \lambda_k q_k q_k^\transpose.
\label{eq:covzee}
\end{equation}}
\rev{It is worth noting that this is not the only way to do this, and it may be possible to use a less invasive 
and/or or less expensive method to guarantee non-negativity of the covariance.  For example, banding \cite{bickel2008regularized}, shrinkage \cite{ledoit2004well},
thresholding \cite{bickel2008covariance}, or localization \cite{anderson2012localization} are some prospective alternatives.  
In particular, it will be necessary to consider such alternatives as the dimension grows and the cost of factorizing $C^{\rm ML}_n$ 
becomes a dominant consideration, but this is outside the scope of the present work.}
In the update 
step 
the multilevel Kalman gain is 
defined as follows
\begin{equation}
 \kMLMC{n+1}  = C^{\rm ML}_{n+1} H^\transpose (S^{\rm ML}_{n+1})^{-1}, \text{ where } 
  S^{\rm ML}_{n+1}  = H \tilde{C}^{\rm ML}_{n+1}  H^\transpose + \Gamma. 
\label{eq:newkay}
\end{equation}
{Next, all particle paths are corrected according to measurements and perturbed observations are 
added:
\begin{equation}\label{eq:upsamps}
\begin{split}
  \yTildeL{\ell}{n+1,i}		& = y_{n+1}^\obs + \etaL_{n+1,i}, \\
  \vL{\ell-1}{n+1}(\omega_{i,\ell})  & =     (I - \kMLMC{n+1} H ) \vHatL{\ell-1}{n+1}(\omega_{i,\ell} ) + \kMLMC{n+1} \yTildeL{\ell}{n+1,i},  \\
  \vL{\ell}{n+1}(\omega_{i,\ell})  & =     (I - \kMLMC{n+1} H ) \vHatL{\ell}{n+1}(\omega_{i,\ell} ) + \kMLMC{n+1} \yTildeL{\ell}{n+1,i},  
  \end{split}
\end{equation}
where 
$\{\etaL_{n+1,i}\}_{i=1}^{M_\ell}$ are i.i.d.~with $\eta^{\{0\}}_{n+1,1} \sim N(0, \Gamma)$.}
It is in this step precisely that the pairs all become correlated with one another and the situation 
becomes significantly more complex than the i.i.d. case.  After the first update, this correlation propagates
forward through \eqref{eq:DlvHatDef} to the next observation time via this ensemble.  This is the conclusion
of the update 
step of the MLEnKF, and this multilevel ensemble is subsequently propagated 
forward to the next prediction 
time via \eqref{eq:DlvHatDef}.

The {\it multilevel} sample mean and covariance 
(in the case that \eqref{eq:covzee} has not modified the covariance, 
i.e. it has all non-negative eigenvalues without truncation) 
of this multilevel ensemble are given by:
\begin{align}
\nonumber
\meanMLMC{n+1} &= \sum_{\ell=0}^L E_{M_\ell}[\DlV{n+1}(\omega_{\ell,\cdot})] \\
& = (I-\kMLMC{n+1} H) \meanHatMLMC{n+1} + \kMLMC{n+1} \left[ E_{M_0}[\yTildeL{0}{n+1,\cdot}  - y_{n+1}^\obs ]   + y_{n+1}^\obs \right],\\
\nonumber
\covMLMC{n+1}  & = (I-\kMLMC{n+1} H) \covHatMLMC{n+1}(I-\kMLMC{n+1} H)^\transpose +
 \kMLMC{n+1} \cov_{M_0}[ \yTildeL{0}{n+1,\cdot}  - y_{n+1}^\obs ] {\kMLMC{n+1}}^\transpose \\
& = (I-\kMLMC{n+1} H) \covHatMLMC{n+1} + \kMLMC{n+1} \left[ \cov_{M_0}[ \yTildeL{0}{n+1,\cdot}  - y_{n+1}^\obs ] -\Gamma \right] {\kMLMC{n+1}}^\transpose.
\label{eq:incidental}
\end{align}
The second term appearing in each case is unbiased.  
For computing general quantities of interest, it is instructive to introduce
the empirical measure of the multilevel ensemble 
$\{[\hhv^{\ell}_{n}(\omega_{\ell,i}), \hhv^{\ell-1}_{n}(\omega_{\ell,i})]_{i=1}^{M_\ell}\}_{\ell=0}^L$, 
i.e.
\footnote{Similar may be done for the predicting distributions, but
  the updated distributions will be our primary interest.}
\begin{equation}
\mu^{\rm ML}_n = \frac{1}{M_0} \sum_{i=1}^{M_0} \delta_{
  \vL{0}{n}(\omega_{0,i})} + 
\sum_{\ell=1}^{L} \frac{1}{M_\ell} 
\sum_{i=1}^{M_\ell}( \delta_{\vL{\ell}{n}(\omega_{\ell,i}) } -
\delta_{\vL{\ell-1}{n}(\omega_{\ell,i})} ).
\label{eq:mlemp}
\end{equation}
Then, the following shorthand notation
for multilevel sample averages can be introduced.  
For any $\varphi: \bbR^d \rightarrow \bbR$, let 
\[
\mu_n^{\rm ML}(\varphi) := \int \varphi d\mu^{\rm ML}_n = 
\sum_{\ell=0}^{L}  \frac{1}{M_\ell}\sum_{i=1}^{M_\ell} 
{\varphi(\vL{\ell}{n}(\omega_{\ell,i})) - \varphi(\vL{\ell-1}{n}(\omega_{\ell,i}))}.
\]

\subsection{Nonlinear Kalman filtering}
\label{sec:mfenkf}

It will be useful to introduce the limiting process, 
in the case of nonlinear non-Gaussian forward model \eqref{eq:psiDefinition}, i.e. nonlinear \eqref{eq:sde}.
The following nonlinear Markov process defines the mean-field EnKF \cite{law2014deterministic}:
\begin{equation}
\qquad\;\;\;\;\quad\quad\mbox{Prediction}\;\left\{\begin{array}{lll}
   {v}_{n+1}& = \Psi ({\widehat{{v}}}_n), \\
{{m}}_{n+1}&=\bbE[{v}_{n+1}],\\
{{C}}_{n+1}&=\bbE[({v}_{n+1}-{m}_{n+1})\otimes({v}_{n+1}-{m}_{n+1})]
 \end{array}\right.
\label{eq:mfpredict}
\end{equation}
\begin{equation}
\mbox{Update}\left\{\begin{array}{llll} 
S_{n+1}&=H{{C}}_{n+1}H^\transpose+\Gamma \\
K_{n+1}&={C}_{n+1}H^\transpose S_{n+1}^{-1}\\
{\tilde y}_{n+1}&=y_{n+1}^\obs+\tilde{\eta}_{n+1}\\
{\widehat{{v}}}_{n+1}&=(I-K_{n+1}H){v}_{n+1}+K_{n+1}{\tilde y}_{n+1}.\\
\end{array}\right.
\label{eq:mfupdate}
\end{equation}
\vspace{4pt}
Here $\{\tilde{\eta}_n\}_{n=1}^N$ are i.i.d. draws from 
$N(0,\Gamma).$  
\kl{The expectations appearing above in \eqref{eq:mfpredict} are with respect to the random variable 
${v}_{n+1}$, which depends upon the randomness from the initial condition ${v}_0=u_0$, 
the maps $\Psi$, and $\tilde{\eta}_0,\dots, \tilde{\eta}_n$.  
The observed value $y^{\obs}$ is considered fixed and is not averaged over.  }
It is easy to verify that in the linear Gaussian case of the Section
\ref{subsec:kalmanFiltering}, the mean and variance of the 
above process correspond to the mean and variance of the filtering distribution.
Furthermore, it was shown in \cite{mandel2011convergence, le2011large} that 
the single level EnKF converges to the Kalman 
filtering distribution with the standard rate $\cO(M^{-1/2})$ in this case, \kl{as stated formally 
in \eqref{eq:empconv}. 
It was furthermore shown in \cite{le2011large} and \cite{law2014deterministic} 
that for nonlinear Gaussian state-space models and fully non-Gaussian models 
\eqref{eq:psiDefinition}, 
respectively, the {\it same convergence property holds}, with the measure corresponding  
to ${v}_n$ in \eqref{eq:mfpredict} and \eqref{eq:mfupdate} 
replacing $\mu_n$ in \eqref{eq:empconv},
as long as the model satisfies a Lipschitz criterion as in Assumption~\ref{ass:psilip}.
In this work, the aim is to show that the MLEnKF converges as well, and with 
a cost-to-$\varepsilon$ which is strictly smaller than its single level EnKF counterpart.
The true filtering distribution of $u_n|(Y_n=Y_n^{\rm obs})$ will not appear in the remainder of 
this work, and the variable ${v}_n$ will correspond to the solution of the above system 
(noting that the two are equivalent in the linear Gaussian case).}

\section{Theoretical Results}
\label{sec:theory}

The approximation error and
computational cost of approximating the true 
filtering distribution by MLEnKF when given a sequence of 
observations $y_1, y_2, \ldots, y_n$ 
will be studied in this section. 
The notation $|\cdot|$ will be used for standard Euclidean norm (and the induced matrix norm)
and the covariance matrix of random variables $Z, X \in \rSet^d$
will be denoted
\[
\cov[Z,X] :=\E{ (Z-\E{Z}) (X-\E{X} )^\transpose },
\]
with the shorthand  $\cov[Z] =  \cov[Z,Z]$.
Before stating the main approximation theorem, it will be useful to 
present the basic assumptions that will be used throughout and the 
corresponding standard MLMC approximation results for i.i.d. samples,
as well as a slight variant which will be useful in what follows.

\begin{assumption} 
\label{ass:mlrates}
Consider the $d$-dimensional SDE~\eqref{eq:sde}
with initial data $u_0 \in \cup_{p \in \nSet} L^p(\Omega)$.
For the hierarchy of solution operators
defined in Section~\ref{sec:kalman}, let 
$\psiL$ denote a numerical solver using a uniform time 
step $\Dt{\ell} = 1/N_\ell$ with 
$N_\ell/N_{\ell-1} \ge \widehat N >1$ for $\ell=0,1,\ldots$. 
Let $\mathcal{F}$ denote the set of functions 
$\varphi : \rSet^d \to \rSet$ 
{which, for all $\ell \geq 0$ and all $u,v \in \cup_{p \in \nSet} L^p(\Omega)$,
and a given set of constants $\alpha, \beta, \gamma>0$ 
with 
$\alpha \geq \min(\beta,\gamma)/2$, 
 fulfill
}
\begin{enumerate}

\item[(i)] 
$\abs{\E{ \varphi(\psiL(u)) -\varphi(\Psi(v)) } }
 \lesssim N_\ell^{-\alpha}$, {and} 
$\abs{\E{\varphi(u) - \varphi(v)}}  \lesssim N_\ell^{-\alpha}$, { provided } 
$\abs{\E{u -v } } \lesssim  N_\ell^{-\alpha} \; ;$
 
\item[(ii)] $ \| \varphi(\psiL(v)) - \varphi(\Psi^{\ell-1}(v)) \|_p \lesssim
  N_\ell^{-\beta/2}$, for all $p\ge 2 \; ;$ 

\item[(iii)] $\costL := \cost{\psiL(v)} \lesssim N_\ell^{\gamma} \; ;$
 \end{enumerate}
\kl{where, as stated above, 
the notation $f(M) \lesssim g(M)$ here is used to denote $f(M) = \cO(g(M)).$}
{
Assume further that all monomials of degree less than 
or equal to $2$, 
are contained in $\mathcal{F}$.}
\end{assumption}

\begin{remark}
An implication of {the above} condition $(ii)$ is that condition
(i) holds with $\alpha = \beta/2$.  However, for many numerical
schemes, there are settings where it is possible to achieve rates
$\alpha > \beta/2$~{(implemetationally, this may
  yield savings in the computational cost).}  
{The literature~\cite[Theorem 14.5.2]{KlPl92} and~\cite[Chapter
    7]{GrahamTalay}} {provide
  sufficient regularity conditions} {on the SDE problem and $\varphi$ 
  for the the Euler--Maruyama method to achieve the rate
exponents $\alpha = 1$ and {$\beta = 1$}, and the Milstein method to achieve 
  $\alpha=1$ and {$\beta=2$}.}
\end{remark}

We will now state the main theorem of this paper. It gives an upper
bound for the computational cost of achieving a sought accuracy in 
$L^p$-norm when using the MLEnKF method to approximate the expectation of an
observable.  The theorem may be considered an extension to the data
assimilation setting of earlier ``one-step''
cost vs.~error results in multilevel Monte Carlo,
cf.~\cite[Theorem 3.1]{Giles08} and~\cite[Theorem 1]{Cliffe11}. 
To reduce the number of repetitions in the below proofs we notice once and for all
that the process itself is 
in $L^p$ by Assumption~\ref{ass:psilip}, 
hence the realization giving rise to the observations $u_n$ 
and the observations themselves $y_n$ are as well, for $n=1,2,\ldots, N$. 
It follows from this and the 
finite norm of $K_n$ 
\cite{le2011large, law2014deterministic, mandel2011convergence}
that the 
elements $\bar{\widehat{v}}_n$ and $\bar{v}_n$ given by \eqref{eq:mfpredict} and \eqref{eq:mfupdate}
are also in $L^p$ 
for $n=1,2,\ldots, N$. 
It will be assumed that the update comes at a marginal cost with respect to the prediction.  
This may be the case for complicated forward solution with small error tolerance, 
large ensemble, and comparably modest dimension $d$.

{\begin{definition}\label{def:localLip}
A function $\varphi: \rSet^d \to \rSet$ is said to be locally Lipschitz
continuous with at most polynomial growth at infinity provided that 
there exist positive scalars $\nu, C_\varphi <\infty$ such that 
\begin{equation}\label{eq:localLipschitz}
\abs{\varphi(x) -\varphi(y)} \leq C_\varphi |x-y|(1+|x|^{\nu} + |y|^{\nu}), \qquad \forall x,y \in \rSet^d.
\end{equation}
\end{definition}}

\kl{The notation $f(M) \eqsim g(M)$ will be used to indicate that
  there exist constants $\tilde c_1, \tilde c_2 >0$ such that $\tilde c_1
  g(M) \leq f(M) \leq \tilde c_2 g(M)$.}
  
\begin{theorem}[MLEnKF accuracy vs. cost]\label{thm:main}
Suppose Assumptions~\ref{ass:psilip} and~\ref{ass:mlrates} hold. 
For a given $\varepsilon >0$, 
let $L$ and $\{M_\ell\}_{\ell=0}^L$ be defined under the constraints
$L \eqsim \log(\varepsilon^{-1})$ 
and 
\begin{equation}\label{eq:chooseMlr}
 M_\ell \eqsim \left \lceil N_\ell^{-\frac{\beta + 2\gamma}{3}}               
 \begin{cases} 
  N^{2\alpha}_L, & \text{if} \quad \beta > \gamma, \\
  L^2N^{2\alpha}_L, & \text{if} \quad  \beta = \gamma, \\
  N^{2\alpha + \frac{2}{3}(\gamma - \beta)}_L, &\text{if} \quad \beta< \gamma. 
 \end{cases}\right\rceil.
\end{equation}
{Then for all functions $\varphi \in \mathcal{F}$ that are locally Lipschitz
  continuous with at most polynomial growth at infinity, cf.~Definition~\ref{def:localLip}},
we have that
\begin{equation}
\|\mu^{\rm ML}_n (\varphi) - \mu_n (\varphi) \|_p \lesssim \hh{\abs{\log(\varepsilon)}^n} \varepsilon,
\label{eq:lperror}
\end{equation}
where $\mu^{\rm ML}_n$ is the multilevel empirical measure defined in
\eqref{eq:mlemp}, where the samples are given by the multilevel predict
\eqref{eq:DlvHatDef} and update \eqref{eq:upsamps} formulae,
approximating the time $t_n=n$ mean-field EnKF distribution $\mu_n$
 (the filtering distribution $\mu_n=N(m_n,C_n)$ in the linear Gaussian case).
 And the computational cost of the MLEnKF estimator over the time
sequence satisfies 
\begin{equation}\label{eq:mlenkfCosts2}
\cost{\mathrm{MLEnKF}} \lesssim 
\begin{cases}
\varepsilon^{-2}, & \text{if} \quad \beta > \gamma,\\              
\varepsilon^{-2} \abs{\log(\varepsilon)}^3, & \text{if} \quad \beta = \gamma,\\
\varepsilon^{-\left( 2 + \frac{\gamma-\beta}{\alpha} \right)}, & \text{if} \quad \beta < \gamma.
\end{cases}
\end{equation}
\end{theorem}

\begin{remark}
  \hh{The growth in error factor $\abs{\log(\varepsilon)}^n$
  in~\eqref{eq:lperror} is due to a
  propagation of perturbed observation errors of the
  MLEnKF estimator that 
  has been conservatively bounded by the triangle
  inequality in~\eqref{eq:lFactor}. In our numerical tests we do however
  \emph{not} observe the error growth factor, and therefore we believe it might be
  possible to 
  eliminate this factor by sharper theoretical bounds.}
\end{remark}

The proof \hh{of Theorem~\ref{thm:main}} follows roughly along the
same lines as that of \cite{le2011large}, however with more notation
\hh{and longer calculations} due to the multilevel aspect. The proof
also has connections to the work~\cite{Chernov14}, in which an MLMC
method is developed for estimation of higher order central moments.

It will be convenient to introduce
the mean-field limiting multilevel ensemble
$\{[\overline{v}^{\ell}_{n}(\omega_{\ell,i}), \overline{v}^{\ell-1}_{n}(\omega_{\ell,i})]_{i=1}^{M_\ell}\}_{\ell=0}^L$, 
\cite{le2011large, law2014deterministic, mandel2011convergence}, which
evolves according to the same equations {\it with the same
  realizations of noise} except the covariance $\hc_n$, hence the
Kalman gain $K_n$, are given by limiting formulae
in~\eqref{eq:mfpredict} and~\eqref{eq:mfupdate}.  That is, the
intra-level pairs of ensemble members
$(\bar{v}_n^{\ell}(\omegaLI), \bar{v}_n^{\ell-1}(\omegaLI))$
are {\it independent and identically distributed} (i.i.d.) over index
$i$, and they are independent between levels.  
An ensemble member 
$(\hat{\bar{v}}_{n}^{\ell}(\omegaLI), \hat{\bar{v}}_{n}^{\ell-1}(\omegaLI))$
maps to 
$({\bar{v}}_{n+1}^{\ell}(\omegaLI),{\bar{v}}_{n+1}^{\ell-1}(\omegaLI))$
as in \eqref{eq:DlvHatDef}.  Then 
$(\hat{\bar{v}}_{n+1}^{\ell}(\omegaLI), \hat{\bar{v}}_{n+1}^{\ell-1}(\omegaLI))$ is obtained as in 
\eqref{eq:upsamps}, except with $K_{n+1}$ from \eqref{eq:mfupdate}
replacing $K_{n+1}^{\rm ML}$ in \eqref{eq:upsamps}.
The noise realizations $\{\omegaLI\}$ are assumed to be the same as
the EnKF ensemble member $({v}_n^{\ell}(\omegaLI), {v}_n^{\ell-1}(\omegaLI))$.
The sole difference is that the
limiting ensemble is independent between levels and the pairs within a
level are i.i.d.  {\it This is because the covariance and gain come
  from the infinite limiting system \eqref{eq:mfpredict} and
  \eqref{eq:mfupdate}.}  The only correlations are between
$\bar{v}_n^{\ell}(\omegaLI)$ and
$\bar{v}_n^{\ell-1}(\omegaLI)$, due to the $\omegaLI $.
Hence there is no multiplicative propagation of correlations within a
level or between levels.  This crucial fact allows to (a) on the one
hand extend standard multilevel theory for i.i.d. draws over multiple
updates, and (b) on the other hand, establish the required proximity
of the two multilevel ensembles particle-wise, based on convergence
of the random gains $K^{\rm ML}_n$ to the deterministic ones $K_n$.
The latter will require the greatest effort and will dominate the
proof by means technical lemmas.  Note that 
$\|\bar{v}^{\ell}_n\|_p$, $\|\bar{\hhv}^{\ell}_n\|_p$,  $|K_n| <\infty$, following from  Assumptions~\ref{ass:psilip}.

The first step is to bound the multilevel predicting covariance in terms of its constituents, the ensemble members.  
The gain is then bounded in terms of the covariance, and ultimately the updated ensemble in terms of the 
predicting ensemble and the covariance.  
The rate appears only by virtue of the convergence of the i.i.d. ensemble covariance, 
and it is propagated forward by induction.
Only the {\it predicting} covariance will be considered and hats will be omitted
to avoid unnecessary notation.

Recall the multilevel Kalman gain is 
defined as follows
\begin{equation}
K^{\rm ML}_n = C^{\rm ML}_n H^\transpose (H \tilde{C}^{\rm ML}_n H^\transpose + \Gamma)^{-1}, 
\nonumber
\end{equation}
where 
\begin{equation}
\tilde{C}^{\rm ML}_n = \sum_{k=1; \lambda_k> 0}^d \lambda_k q_k q_k^\transpose,
\label{eq:covzee1}
\end{equation}
for eigenpairs $\{\lambda_k,q_k\}$ of $C^{\rm ML}_n$.
The following micro-lemma will be necessary to control the error in the gain.

\begin{lemma}[multilevel covariance approximation error]  
\label{lem:mlcae}
Let $\tilde{C}^{\rm ML}_n$ be given by~\eqref{eq:covzee1}. Then 
the following 
holds 
\begin{equation}
  |\tilde{C}^{\rm ML}_n - {C}^{\rm ML}_n| \leq |C^{\rm ML}_n - C_n|,
  \label{eq:mlcov_bound}
\end{equation}
\kl{where $|\cdot|$ denotes the induced 2-norm for matrices.}
\end{lemma}

\begin{proof}
Notice \kl{that}
\begin{equation}
|\tilde{C}^{\rm ML}_n - {C}^{\rm ML}_n| =  {\rm max}_{\{j; \lambda_j<0\}} \{|\lambda_j|\}.
\end{equation}
Denote the associated eigenvector by $u_{\rm max}$ (normalized to $|u_{\rm max}|=1$).
\kl{Notice that for any $A=A^\transpose$,
\[
|A| = {\rm sup}_u \frac{|u^\transpose A u|}{|u|^2} = {\rm max}_k |\lambda_k|,
\]
where $\lambda_k$ are the eigenvalues of $A$.} Since $C_n \geq 0$, one has that 
\[
| u_{\rm max}^\transpose ({C}^{\rm ML}_n - C_n) u_{\rm max} | = u_{\rm max}^\transpose C_n u_{\rm max}
-u_{\rm max}^\transpose {C}^{\rm ML}_n u_{\rm max} \geq |\tilde{C}^{\rm ML}_n - {C}^{\rm ML}_n|.
\]

\end{proof}

The next step is to bound the gain error, which is done in the following lemma.
\begin{lemma}[Continuity of the gain in the covariance]
\label{lem:gce}
There is a constant $c_n<\infty$, depending on $|H|, \gamma_{\min},$ and $|K_n H|$
such that 
\begin{equation}
|K^{\rm ML}_n - K_n| \leq c_n |{C}^{\rm ML}_{n}-C_{n}|, 
\label{eq:gaincov}
\end{equation}
\kl{where $\gamma_{\rm min} >0$ is the smallest eigenvalue of $\Gamma$.}
\end{lemma}

\begin{proof}
Recall that
\begin{equation}\label{eq:kaytri}
\begin{split} 
K_n - \kMLMC{n} &=  C_n H^\transpose \left ( (HC_n H^\transpose + \Gamma)^{-1}  - (H\tilde{C}^{\rm ML}_{n} H^\transpose + \Gamma)^{-1} \right )  \\
  & \quad +  (C_{n} - {C}^{\rm ML}_{n})H^\transpose (H \tilde{C}^{\rm ML}_{n} H^\transpose + \Gamma)^{-1},
\end{split}
\end{equation}
where $\tilde{C}^{\rm ML}_{n} \geq 0$ is defined in \eqref{eq:covzee}, and notice that
\begin{multline}
(H C_n H^\transpose + \Gamma)^{-1}  - (H\tilde{C}^{\rm ML}_{n} H^\transpose + \Gamma)^{-1}\\ = 
(HC_n H^\transpose +\Gamma)^{-1}H (\tilde{C}^{\rm ML}_{n} -C_n) H^\transpose(H\tilde{C}^{\rm ML}_{n} H^\transpose +\Gamma) ^{-1}.
\end{multline}
So 
\begin{eqnarray*}
K_n - \kMLMC{n} &=&   K_n H (\tilde{C}^{\rm ML}_{n} -C_n) H^\transpose (H \tilde{C}^{\rm ML}_{n}H^\transpose + \Gamma)^{-1} \\
&+& (C_{n} - {C}^{\rm ML}_{n})H^\transpose (H \tilde{C}^{\rm ML}_{n} H^\transpose + \Gamma)^{-1}. 
\label{eq:kay}
\end{eqnarray*}
Note that $x^\transpose (\Gamma+B) x \geq x^\transpose \Gamma x \geq \gamma_{\rm min}$ 
for all $x \in \bbR^d$ whenver $B=B^\transpose \geq 0$, and this implies that 
$|(H \tilde{C}^{\rm ML}_{n} H^\transpose + \Gamma)^{-1}| \leq 1/{\gamma_{\rm min}}$.
It follows by \eqref{eq:mlcov_bound} that 
\begin{equation*}\label{eq:kayer1}
|K_n - \kMLMC{n}| \leq \frac{|H|}{\gamma_{\rm min}}(1+ 2|K_n H|) | C_{n} - {C}^{\rm ML}_{n}|.
\end{equation*}
\end{proof}

\rev{It is worth noting that the multilevel gain error is bounded by the 
{\it unmodified} multilevel sample covariance error, following from Lemma \ref{lem:mlcae}, 
so modification in \eqref{eq:covzee1} will not affect the ultimate approximation error.}

\begin{theorem}
\label{thm:covspliteps}
Suppose Assumptions~\ref{ass:psilip} and~\ref{ass:mlrates} hold.  
For any $\varepsilon>0$, 
let $L$ and $\{M_\ell\}_{\ell=0}^L$ be defined as in 
Theorem~\ref{thm:main}.
Then the following \kl{asymptotic} inequality holds
\begin{equation}
\|C^{\rm ML}_n - C_n\|_p \lesssim \varepsilon 
+ \|C^{\rm ML}_n - \bar{C}^{\rm ML}_n\|_p
\label{eq:covspliteps}
\end{equation}
with a cost which satisfies 
\eqref{eq:mlenkfCosts2}.
\end{theorem}

\begin{proof}
Let $C^L_n$ denote 
the predicting covariance of the final $L^{th}$ level limiting system at time $n$, 
in the sense that the forward map above is replaced by $\Psi^L$, 
but the gain comes from the continuum mean-field limiting system.
Furthermore, let $\bar{C}^{\rm ML}_n$ denote the covariance associated to
the multilevel ensemble $\{(\bar{v}^{\ell}_{n,i})_{i=1}^{M_\ell}\}_{\ell=1}^L$.
The triangle inequality is used to split
\begin{equation}
|C^{\rm ML}_n - C_n| \leq |C^L_n - C_n| + |\bar{C}^{\rm ML}_n - C^L_n| + |C^{\rm ML}_n - \bar{C}^{\rm ML}_n|,
\label{eq:covsplit}
\end{equation}
and each term will be dealt with in turn, in the following three
lemmas.  The proof of the theorem is 
done after establishing 
\hh{Lemmas~\ref{lem:boundsBarV}},~\ref{lem:disccov} and~\ref{lem:iidcover},
which provide the asymptotic bound on the first two terms.
\end{proof}

\hh{
\begin{lemma}
  \label{lem:boundsBarV}
  Suppose Assumptions~\ref{ass:psilip} and~\ref{ass:mlrates} hold.  
  For any $\varepsilon>0$, let $L$ and $\{M_\ell\}_{\ell=0}^L$ be defined as in 
  Theorem~\ref{thm:main}. Then, for any finite $n$, 
  \begin{equation}\label{eq:barVWeak}
      \max\parenthesis{\abs{\E{\bar\hv_n^L - \hv_n}}, \abs{\E{\bar\hhv_n^L - \hhv_n}}} \lesssim \varepsilon, 
  \end{equation}
  and for $\ell =0,1,\ldots, L$ and all $p\ge 2$, 
  \begin{equation}\label{eq:barVStrong}
      \max\parenthesis{\norm{\bar \hv_n^\ell - \bar \hv_n^{\ell-1}}_p, \norm{\bar \hhv_n^\ell - \bar \hhv_n^{\ell-1}}_p}  \lesssim N_\ell^{-\beta/2}.
  \end{equation}

\end{lemma}

\begin{proof}
  Since 
  the initial data $\bar\hhv_0^L=\bar v_0^L=\hhv_0=v_0$
  is the same,
\[
\abs{\E{\bar\hhv_0^L - \hhv_0}} =0.
\]
Now assume that the following holds
\[
\max\parenthesis{\abs{\abs{\E{\bar\hhv_{n-1}^L - \hhv_{n-1}}}}, \abs{\E{\bar\hv_{n-1}^L - \hv_{n-1}}}} \lesssim \varepsilon
\]
Assumption~\ref{ass:mlrates} (i) directly 
implies 
\[
\abs{\E{\bar\hv_{n}^L - \hv_{n}}} \lesssim \varepsilon.
\]
Futhermore, since $|K_nH|<\infty$ 
for any finite $n$, 
\[
\abs{\E{\bar\hhv_{n}^L - \hhv_{n}}} \leq |I - K_nH| \abs{\E{\bar\hv_n^L - \hv_n}} \lesssim \varepsilon,
\]
and inequality~\eqref{eq:barVWeak} follows by induction.

To prove inequality~\eqref{eq:barVStrong}, recall that due to the matching initial data,
the inequality holds trivially at $n=0$.
Assume 
\[
\max\parenthesis{\norm{\bar \hv_{n-1}^\ell - \bar \hv_{n-1}^{\ell-1}}_p, \norm{\bar \hhv_{n-1}^\ell - \bar \hhv_{n-1}^{\ell-1}}_p}  \lesssim N_\ell^{-\beta/2}.
\]
By Assumptions~\ref{ass:psilip} (i) and~\ref{ass:mlrates} (ii), 
\[
\begin{split}
\norm{\bar \hv_{n}^\ell - \bar \hv_{n}^{\ell-1}}_p
& \leq \norm{ \Psi^\ell(\bar \hhv_{n-1}^\ell) - \Psi^{\ell-1}(\bar \hhv_{n-1}^\ell)}_p
+ \norm{ \Psi^{\ell-1}(\bar \hhv_{n-1}^\ell) - \Psi^{\ell-1}(\bar
  \hhv_{n-1}^{\ell-1})}_p \\
& \lesssim N_\ell^{-\beta/2},
\end{split}
\]
\[
\norm{\bar \hhv_{n}^\ell - \bar \hhv_{n}^{\ell-1}}_p \lesssim |I - K_nH| \norm{\bar\hv_n^\ell - \bar \hv_n^{\ell-1} }_p \lesssim N_\ell^{-\beta/2},
\]
and inequality~\eqref{eq:barVStrong} holds by induction.
\end{proof}
}

\begin{lemma}[Covariance discretization error]
\label{lem:disccov}
Suppose Assumptions~\ref{ass:psilip} and~\ref{ass:mlrates} hold. For any 
$\varepsilon>0$, let $L$ be defined as in Theorem~\ref{thm:main}.
Then the following \kl{asymptotic inequality} holds
\begin{align}
|C^L_n - C_n| \lesssim \varepsilon.
\label{eq:disccov}
\end{align}
\end{lemma}
\begin{proof}
It is possible to show that for any symmetric matrix $A$, the following inequality holds
\begin{equation}
|A| \leq \sum_{j,j'=1}^d |A^{jj'}|.
\label{eq:l2l1matnorm}
\end{equation}
Furthermore, by adding the terms 
$ \pm  \E{ ({\hv}_n)^{j}} \bbE\big[(\bar{\hv}^L_n)^{j'}\big]$,
\begin{align}
\nonumber
|(C^L_n - C_n)^{jj'}| & = \Big|\E{(\bar{\hv}^L_n)^j (\bar{\hv}^L_n)^{j'}}  - \E{({\hv}_n)^j ({\hv}_n)^{j'}} \\
& -    \E{ (\bar{\hv}^L_n)^j} \E{(\bar{\hv}^L_n)^{j'}} + \E{ ({\hv}_n)^{j}} \E{({\hv}_n)^{j'}}\Big| \\
\nonumber
& \leq  \abs{\E{(\bar{\hv}^L_n)^j (\bar{\hv}^L_n)^{j'}  -  ({\hv}_n)^j ({\hv}_n)^{j'}}} \\
\nonumber
& + \abs{\E{(\bar{\hv}^L_n)^{j'}} } \abs{\E{(\bar{\hv}^L_n)^{j} -  ({\hv}_n)^{j}}} + 
|\E{({\hv}_n)^{j}}| \abs{ \E{(\bar{\hv}^L_n)^{j'} - ({\hv}_n)^{j'}}} \\
\nonumber
& \lesssim N_L^{-\alpha} \eqsim \varepsilon.
\end{align}
The last inequality follows by 
\kl{Lemma~\ref{lem:boundsBarV} and Assumption~\ref{ass:mlrates} (i), 
noting that $\mathcal{F}$ contains all monomials of degree less than or equal to 2.  }

\end{proof}

Notice that 
$$
C^L_n =\sum_{\ell=0}^L \cov[\bar{v}^\ell_n] - \cov[\bar{v}^{\ell-1}_n],
$$
and
$$
C^{\rm ML}_n =  
\sum_{\ell=0}^L \cov_{M_\ell}[v^\ell_n] - \cov_{M_\ell}[v^{\ell-1}_n],
$$
with the convention that $v^{-1} = \bar{v}^{-1} := 0$. 
Consider also the partner covariance to the above
$$
\bar{C}^{\rm ML}_n = \sum_{\ell=0}^L \cov_{M_\ell}[\bar{v}^\ell_n] - \cov_{M_\ell}[\bar{v}^{\ell-1}_n].
$$
The next two differences are bounded in terms of the single-level differences, 
using the triangle inequality to extend to the sum.

\begin{lemma}[multilevel i.i.d. sample covariance error]  
\label{lem:iidcover}
Suppose Assumptions~\ref{ass:psilip} and~\ref{ass:mlrates} hold, and
for any $\varepsilon>0$,  
let $L $ and $\{M_\ell\}_{\ell=0}^L$ be defined as in Theorem~\ref{thm:main}.
Then the following \kl{asymptotic inequality holds}
\begin{equation}
\|\bar{C}^{\rm ML}_n - C^L_n\|_p  \lesssim \varepsilon.
\label{eq:iidcover}
\end{equation}
\end{lemma}

\begin{proof}
Notice the following triangle inequality
\[
\begin{split}
 |\bar{C}^{\rm ML}_n - C^L_n| & \leq 
 \sum_{\ell=0}^L |\cov_{M_\ell}[\bar{\hv}^\ell_n] - \cov_{M_\ell}[\bar{\hv}^{\ell-1}_n] 
- ( \cov_{M_\ell}[\bar{\hv}^\ell_n] - \cov_{M_\ell}[\bar{\hv}^{\ell-1}_n])|.
\end{split}
\]
To avoid needlessly long terms when bounding the summands of the
above equation, we now make the 
assumption in this proof that $\E{\bar{\hv}^\ell_n} = 0$, without loss of generality. 
We may then obtain the rearrangement
\[
\cov_{M_{\ell}}[\bar{\hv}^\ell_n] = E_{M_\ell}\big[ \bar{\hv}^\ell_n(\bar{\hv}^\ell_n)^\transpose\big] - 
E_{M_\ell}[\bar{\hv}^\ell_n] \big(E_{M_\ell}[\bar{\hv}^\ell_n]\big)^\transpose,
\]
and similarly for the $\ell-1$ term.  
Using the identity $ aa^\transpose - b b^\transpose = \frac12 [ (a+b)(a-b)^\transpose + (a-b)(a+b)^\transpose ]$ for $a,b \in \bbR^d$
on each of the outer products with $\ell, \ell-1$, respectively, 
and then using \eqref{eq:l2l1matnorm} again for the first term, 
and Cauchy-Schwartz for the second (and grouping like terms arising from the $(j,j') \rightarrow (j',j)$ 
symmetry of $\frac12(a^jb^{j'} + a^{j'}b^j)$), one has 
\begin{align}
\nonumber 
& \abs{\cov_{M_\ell}[\bar{\hv}^\ell_n] - \cov_{M_\ell}[\bar{\hv}^{\ell-1}_n] - ( \cov[\bar{\hv}^\ell_n] - \cov[\bar{\hv}^{\ell-1}_n])} \\ 
\nonumber
&\leq \sum_{j\leq j'=1}^d \Big | E_{M_\ell}\big[ \big(\bar{\hv}^\ell_n + \bar{\hv}^{\ell-1}_n \big)^j 
\big(\bar{\hv}^\ell_n - \bar{\hv}^{\ell-1}_n \big)^{j'}\big] 
- \E{ \big(\bar{\hv}^\ell_n + \bar{\hv}^{\ell-1}_n \big)^j
  \big(\bar{\hv}^\ell_n - \bar{\hv}^{\ell-1}_n \big)^{j'}} \Big|\\
\nonumber
&+
\Big | E_{M_\ell}\big[ \big(\bar{\hv}^\ell_n + \bar{\hv}^{\ell-1}_n \big)^{j '}
\big(\bar{\hv}^\ell_n - \bar{\hv}^{\ell-1}_n \big)^{j}\big] - 
\E{\big(\bar{\hv}^\ell_n + \bar{\hv}^{\ell-1}_n \big)^{j '}
\big(\bar{\hv}^\ell_n - \bar{\hv}^{\ell-1}_n \big)^{j}}\Big | \\
\nonumber
& + \abs{E_{M_\ell}[\bar{\hv}^\ell_n+ \bar{\hv}^{\ell-1}_n)^j]}
\abs{E_{M_\ell}[(\bar{\hv}^\ell_n- \bar{\hv}^{\ell-1}_n)^{j'}]} \\
& + \abs{E_{M_\ell}[(\bar{\hv}^\ell_n+ \bar{\hv}^{\ell-1}_n)^{j'}]}
\abs{E_{M_\ell}[(\bar{\hv}^\ell_n- \bar{\hv}^{\ell-1}_n)^{j}]}.
\label{eq:ccons}
\end{align}
Almost sure convergence follows by the law of large numbers.  The rate in $L^p$ is 
shown now.

First, it will be necessary to recall the Marcinkiewicz-Zygmund inequality: for i.i.d. random variables $X_1,\ldots,X_N \sim X$
with $ \|X\|_p<\infty$ 
for $p\geq 2$, and $\E{X} =0$,
\begin{equation}\label{eq:mzin0}
\|E_{N}[X]\|_p \leq c_p N^{-1/2} \|X\|_p,
\end{equation}
where the constant depends only on $p$, cf.~\cite{cappe2005inference,
  gut2005probability}; in fact, $c_p \leq 3\sqrt{2p}$, cf.~\cite{Ren01}.

Using the Marcinkiewicz-Zygmund inequality 
then H{\"o}lder's inequality on each of the first two terms on the right-hand side of 
\eqref{eq:ccons}, then the reverse order on the last two, and finally
the Assumptions \ref{ass:mlrates}, Lemma~\ref{lem:boundsBarV}, \kl{and the fact $\bar{v}^\ell_n \in L^p(\Omega)$ for all $p \ge 2$,
together yield}
\begin{equation}
\begin{split}
& \| \cov_{M_\ell}[\bar{\hv}^\ell_n] - \cov_{M_\ell}[\bar{\hv}^{\ell-1}_n] - 
( \cov[\bar{\hv}^\ell_n] - \cov[\bar{\hv}^{\ell-1}_n])\|_{p} \leq \sum_{j\leq j'=1}^d 
M_\ell^{-1/2} \Bigg[ \\
& c_p\left( \| \big(\bar{\hv}^\ell_n + \bar{\hv}^{\ell-1}_n \big)^j \|_{2p} 
\| \big(\bar{\hv}^\ell_n - \bar{\hv}^{\ell-1}_n\big)^{j'}\|_{2p} 
 + \| \big(\bar{\hv}^\ell_n + \bar{\hv}^{\ell-1}_n \big)^{j'}\|_{2p}
\|  \big( \bar{\hv}^\ell_n - \bar{\hv}^{\ell-1}_n )\big)^{j}\|_{2p} \right)\\
&  +c_{2p}^2 M_\ell^{-1/2} \Big( \| \big(\bar{\hv}^\ell_n + \bar{\hv}^{\ell-1}_n\big)^j \|_{2p}
\| \big( \bar{\hv}^\ell_n - \bar{\hv}^{\ell-1}_n\big)^{j'} \|_{2p} + \| \big(\bar{\hv}^\ell_n + \bar{\hv}^{\ell-1}_n)^{j'} \|_{2p}
\| \big(\bar{\hv}^\ell_n - \bar{\hv}^{\ell-1}_n \big)^{j}\|_{2p} \Big) \Bigg]\\
& \lesssim M_\ell^{-1/2} N_\ell^{-\beta/2}. 
\end{split}
\end{equation}

Finally, by the triangle inequality, the following bound holds for 
\eqref{eq:iidcover} for all $p\geq 2$,
\begin{align} 
 \nonumber 
 & \|\bar{C}^{\rm ML}_n - C^L_n\|_p 
\label{eq:coviidbound}
 \lesssim 
\sum_{\ell=0}^L  M_\ell^{-1/2} 
N_\ell^{-\beta/2}  
\lesssim \varepsilon.
\end{align}
\end{proof}

The previous two lemmas complete the proof of Theorem~\ref{thm:covspliteps}.
Now we turn to the next term in \eqref{eq:covspliteps}, the difference between 
multilevel ensemble covariances, which is continuous in the individual ensemble members.
First it will be necessary to recall (see e.g. Lemma 4.3 of~\cite{le2011large}) 
that for identically distributed random variables $x_1,\ldots, x_N \in \bbR^d$,
\begin{equation}
\Big ( \bbE \Big[ \abs{
E_{N}\big[|x_n|^p\big]^{1/p}  }^q \Big] \Big)^{1/q} \leq \|x_n\|_r,
\label{eq:exsum}
\end{equation} 
where $r = {\rm max}\{q,p\}$.

\begin{lemma}[Continuity of multilevel sample covariances in particles]
\label{lem:mlcov}
Suppose Assumptions~\ref{ass:psilip} and~\ref{ass:mlrates} hold, and
for any $\varepsilon>0$,  
let $L$ and $\{M_\ell\}_{\ell=0}^L$ be defined as in Theorem~\ref{thm:main}.
Then the following \hh{asymptotic inequality} holds for all $p\geq 2$,
\begin{equation}
\begin{split}
 \|C^{\rm ML}_n - \bar{C}^{\rm ML}_n\|_p \leq &\sum_{l=0}^L 
\Bigg(  \|\hv_{n}^{\ell} - \bar{\hv}_{n}^\ell\|_{p} + 
4\|\hv_{n}^{\ell} - \bar{\hv}_{n}^\ell\|_{2p} \|\bar{\hv}_{n}^{\ell}\|_{2p}  \\
& + \|\hv_{n}^{\ell-1} - \bar{\hv}_{n}^{\ell-1}\|_{p}
 + 4\|\hv_{n}^{\ell-1} - \bar{\hv}_{n}^{\ell-1}\|_{2p} \|\bar{\hv}_{n}^{\ell-1}\|_{2p}  \Bigg).
\label{eq:mlcov}
\end{split}
\end{equation}
\end{lemma}

\begin{proof}
Recall first that 
\[
\abs{C^{\rm ML}_n - \bar{C}^{\rm ML}_n} \leq 
 \sum_{\ell=1}^L \abs{\cov_{M_\ell}[\hv^\ell_n] - \cov_{M_\ell}[\hv^{\ell-1}_n] - 
(\cov_{M_\ell}[\bar{\hv}^\ell_n] - \cov_{M_\ell}[\bar{\hv}^{\ell-1}_n])}. 
\]
Now the individual terms will be bounded.
Note that
\[
\cov_{M_\ell}[\hv^\ell_n] = E_{M_\ell}[ \hv^\ell_n(\hv^\ell_n)^\transpose] - E_{M_\ell}[\hv^\ell_n] \big(E_{M_\ell}[\hv^\ell_n]\big)^\transpose,
\]
and similar for $\bar{\hv}^\ell_n$.
Using $a^2-b^2 = (a-b)^2 + 2 b (a-b)$ with $a=u^\transpose\hv_n^{\ell}$ and $b= u^\transpose \bar{\hv}_n^\ell$ and again with
 $a=(1/M_\ell)\sum_{i=1}^{M_\ell} u^\transpose\hv_{n,i}^{\ell}$ and 
 $b= (1/M_\ell)\sum_{i=1}^{M_\ell}u^\transpose\bar{\hv}_{n,i}^\ell$ for arbitrary $u \in \bbR^d$, 
these terms are rearranged as follows
\begin{align}
\nonumber
& u^\transpose \big(\cov_{M_\ell}[\hv^\ell_n] - \cov_{M_\ell}[\bar{\hv}^\ell_n] \big) u \\
\nonumber
 & = E_{M_\ell}[\abs{u^\transpose (\hv^\ell_n-\bar{\hv}^\ell_n)}^2] 
+ 2E_{M_\ell}[(u^\transpose\bar{\hv}^\ell_n)(u^\transpose ({\hv}^\ell_n-\bar{\hv}^\ell_n))] \\
\nonumber
& - \abs{E_{M_\ell}[u^\transpose (\hv^\ell_n-\bar{\hv}^\ell_n)]}^2 
- 2E_{M_\ell}[u^\transpose\bar{\hv}^\ell_n] E_{M_\ell}[u^\transpose ({\hv}^\ell_n-\bar{\hv}^\ell_n)]. 
\end{align}
Then, using the Cauchy-Schwartz inequality, 
the first term of \eqref{eq:covsplit} is bounded as follows
\begin{align}
\nonumber
& \abs{\cov_{M_\ell}[\hv^\ell_n] - \cov_{M_\ell}[\hv^{\ell-1}_n] - (\cov_{M_\ell}[\bar{\hv}^\ell_n]
- \cov_{M_\ell}[\bar{\hv}^{\ell-1}_n])} \\
\nonumber
& =  \abs{\cov_{M_\ell}[\hv^\ell_n] - \cov_{M_\ell}[\bar{\hv}^\ell_n] + \cov_{M_\ell}[\hv^{\ell-1}_n] - \cov_{M_\ell}[\bar{\hv}^{\ell-1}_n]} \\
\nonumber
& \leq  E_{M_\ell}\big[\abs{\hv_{n}^{\ell} - \bar{\hv}_{n}^\ell}^2\big] + 4 \sqrt{E_{M_\ell} \big[ \abs{\bar{\hv}_{n}^{\ell}}^2\big] 
E_{M_\ell}\big[\abs{\hv_{n}^{\ell}- \bar{\hv}_{n}^{\ell}}^2\big]} \\
& +  E_{M_\ell} \big[ \abs{\hv_{n}^{\ell-1} - \bar{\hv}_{n}^{\ell-1}}^2 \big] + 
4 \sqrt{E_{M_\ell} \big[\abs{\bar{\hv}_{n}^{\ell-1}}^2]E_{M_\ell} [ \abs{\hv_{n}^{\ell-1}- \bar{\hv}_{n}^{\ell-1}}^2\big]}.
\label{eq:cens}
\end{align}
After rearrangement, the triangle inequality, 
\eqref{eq:exsum} with $p=\max\{p,2\}$, and H{\"o}lder's inequality complete the proof.

\end{proof}

It has just been shown that the second term of~\eqref{eq:covspliteps} is ``close in the predicting ensembles".
Therefore, the error level of the first term will carry over between observation times by induction.  
This is made rigorous by the next lemma.

\hh{
\begin{lemma}[Distance between ensembles.] 
\label{lem:ensdist}
Suppose Assumptions~\ref{ass:psilip} and~\ref{ass:mlrates} hold, and
for any $\varepsilon>0$,  
let $L$ and $\{M_\ell\}_{\ell=0}^L$ be defined as in Theorem~\ref{thm:main}.
Then the following asymptotic inequality holds for all 
$p\geq 2$, 
  \begin{equation}
    \sum_{\ell=0}^L \norm{\hhv_n^\ell - \bar \hhv_n^\ell}_p
    \lesssim \abs{\log(\varepsilon)}^{n} \varepsilon.
    \label{eq:ensdist}
  \end{equation}
\end{lemma}

\begin{proof}
First recall that the assertion holds trivially for $n=0$.  Proceeding
by induction, assume for $p\geq 2$,
\[
\sum_{\ell=0}^L \norm{\hhv_{n-1}^\ell - \bar \hhv_{n-1}^\ell}_p
\lesssim \abs{\log(\varepsilon)}^{n-1} \varepsilon,
\]
Then Assumption~\ref{ass:psilip}(i) implies the following inequality holds for the prediction 
\begin{equation}\label{eq:indlip1ens}
    \sum_{\ell=0}^L \norm{\hv_n^\ell - \bar \hv_n^\ell}_p \leq
    c_{\Psi} \sum_{\ell=0}^L \norm{\hhv_{n-1}^\ell - \bar
      \hhv_{n-1}^\ell}_p \lesssim \abs{\log(\varepsilon)}^{n-1}
    \varepsilon.
\end{equation}
Using Lemma \ref{lem:gce}, 
the following inequalities hold for $\ell=0, \ldots, L$,
\begin{align}\label{eq:appearanceOfCn}
& |\hhv_n^\ell - \bar{\hhv}_n^\ell|  \leq   
 |I-K_n H| |\hv^\ell_n - \bar{\hv}^\ell_n| 
  +  c_n  |C^{\rm ML}_n - {C}_n| \Big(|\hv^\ell_n - \bar{\hv}^\ell_n| + |y_n^\ell - H\bar\hv_n^\ell|\Big).
\end{align}
By H{\"o}lder's inequality and since $y_n^\ell, \bar\hv_n^\ell \in L^p(\Omega)$ for all $p \ge 2$,
\begin{equation*}
\begin{split}
 \|\hhv_n^\ell  -  \bar{\hhv}_n^\ell\|_{p}  &\leq   
 |I-K_n H| \|\hv_n^\ell -  \bar{\hv}_n^\ell\|_p \\
 & +  c_n  \|{C}^{\rm ML}_n - C_n\|_{2p} \Big ( \|\hv_n^\ell - \bar{\hv}_n^\ell \|_{2p} + \|y_n^\ell - H\bar{\hv}_n^\ell\|_{2p} \Big ) \\
& \lesssim 
\| \hv^\ell_n - \bar{\hv}^\ell_n\|_p 
+  \|{C}^{\rm ML}_n - C_n\|_{2p}
\Big ( \| \hv_n^\ell - \bar{\hv}_n^\ell \|_{2p} + 1 \Big ).
\end{split}
\end{equation*}
Plugging the moment bound~\eqref{eq:indlip1ens} 
into the right-hand side of the inequality~\eqref{eq:mlcov}
yields that
$\|C^{\rm ML}_n - \bar{C}^{\rm ML}_n\|_{2p} \lesssim \abs{\log(\varepsilon)}^{n-1}\varepsilon$,
which in combination Theorem~\ref{thm:covspliteps}
further leads to
$\|C^{\rm ML}_n - {C}_n\|_{2p} \lesssim \abs{\log(\varepsilon)}^{n-1}\varepsilon$.
Therefore, summing the above and using \eqref{eq:indlip1ens} again for $p, 2p$ 
\begin{equation}\label{eq:lFactor}
\begin{split}  
 \sum_{\ell=0}^L \|\hhv_n^{\ell} - \bar{\hhv}_n^\ell\|_{p}   & \lesssim 
\sum_{\ell=0}^L  \|\hv_n^{\ell} - \bar{\hv}_n^\ell\|_p  
 + \varepsilon  \left( \|\hv_n^{\ell} -    \bar{\hv}_n^\ell\|_{2p} +   \|y_n^{\ell} - \bar{\hv}_n^\ell\|_{2p} \right) \\
& \lesssim \abs{\log(\varepsilon)}^{n-1} \varepsilon \Big(1+ \sum_{\ell=0}^L\|y_n^{\ell} -
 \bar{\hv}_n^\ell\|_{2p}\Big)\\
 & \lesssim \abs{\log(\varepsilon)}^{n} \varepsilon,
\end{split}
\end{equation}
where the last inequality of the proof uses that $\|y_n^{\ell} -  H \bar{\hv}_n^\ell\|_{2p} \lesssim 1$
and $L \eqsim \abs{\log(\varepsilon)}$.

\end{proof}}

Induction is complete on the distance between the multilevel ensemble
and its i.i.d.~shadow in $L^p$, and it remains only to close the
argument, which is done next. Note that the induction actually holds
for all $n$, but we are able to neglect the $n$-dependence of the
constant $c_n$ appearing in \eqref{eq:appearanceOfCn} by considering only a finite
number $N$ of steps.

\begin{proof}[Proof of Theorem~\ref{thm:main}] 

  What remains is to verify that provided $L$ and $M_\ell$ are defined
  under the constraints in Theorem~\ref{thm:main}, the error
  bound~\eqref{eq:lperror} will be obtained {for all the
  functions $\varphi \in \mathcal{F}$ which 
  are locally Lipschitz continuous with at most polynomial growth
  at infinity, cf.~Definition~\ref{def:localLip}. }
 Notice that the triangle inequality gives
\begin{align}
\nonumber
\|\mu^{\rm ML}_n (\varphi) - \mu_n (\varphi)\|_p & \leq 
\|\mu^{\rm ML}_n (\varphi) - \bar{\mu}^{\rm ML}_n (\varphi)\|_p  \\
& + \|\bar{\mu}^{\rm ML}_n (\varphi) - \bar{\mu}^L_n(\varphi)\|_p 
+ \|\bar{\mu}^L_n(\varphi) - \mu_n (\varphi)\|_p,
\label{eq:lpertri}
\end{align}
where $\bar{\mu}^{\rm ML}_n$ denotes the empirical measure associated to 
the i.i.d. ensemble, and $\bar{\mu}^L_n$ denotes the probability measure 
associated to $\bar{v}^L$. Before treating each term separately, we notice that the 
two first summands of the right-hand side of 
the inequality relates to the statistical error, whereas the last relates to the bias.

The first summand of~\eqref{eq:lpertri} satisfies the following bound
{\begin{align}
\nonumber
\|\mu^{\rm ML}_n (\varphi) - \bar{\mu}^{\rm ML}_n (\varphi)\|_p  & =
 \left\| \sum_{\ell=0}^L E_{M_\ell}\Big[\varphi(\hhv_n^{\ell}) - \varphi(\hhv_n^{\ell-1}) - 
(\varphi(\bar{\hhv}_n^\ell) - \varphi(\bar{\hhv}_n^{\ell-1})) \Big] \right\|_p\\
\nonumber
& \leq
\sum_{\ell=0}^L \|\varphi(\hhv_n^{\ell}) - \varphi(\bar{\hhv}_n^\ell)\|_{p}
+ \|\varphi(\hhv_n^{\ell-1}) - \varphi(\bar{\hhv}_n^{\ell-1})\|_{p} \\
\nonumber
& \leq C \sum_{\ell=0}^L \Big[ 
 \big \||\hhv_n^{\ell} - \bar{\hhv}_n^\ell | \parenthesis{1+|\bar{\hhv}_n^\ell|^{\nu}} 
 + |\hhv_n^{\ell} - \bar{\hhv}_n^\ell|^{\nu+1} \big \|_{p}   \\
\nonumber
 & + \big \||\hhv_n^{\ell-1} - \bar{\hhv}_n^{\ell-1} | \parenthesis{1+|\bar{\hhv}_n^{\ell-1} |^{\nu}} 
 + |\hhv_n^{\ell-1} - \bar{\hhv}_n^{\ell-1} |^{\nu+1}  \big \|_{p}
  \Big]\\
\nonumber
  & \leq C \sum_{\ell=0}^L \Big[ 
 \|\hhv_n^{\ell} - \bar{\hhv}_n^\ell \|_{2p} \parenthesis{1+\|\bar{\hhv}_n^\ell\|_{2p\nu}^{\nu}} 
 + \|\hhv_n^{\ell} - \bar{\hhv}_n^\ell \|_{p(\nu+1)}^{\nu+1}   \\
\nonumber
 & +  \|\hhv_n^{\ell-1} - \bar{\hhv}_n^{\ell-1} \|_{2p} \parenthesis{1+\|\bar{\hhv}_n^{\ell-1}\|_{2p\nu}^{\nu}} 
 + \|\hhv_n^{\ell-1} - \bar{\hhv}_n^{\ell-1} \|_{p(\nu+1)} ^{\nu+1} 
  \Big]\\  
 & \lesssim \abs{\log(\varepsilon)}^n \varepsilon.
 \label{eq:finensembles}
\end{align}
The second inequality follows directly from
the expression~\eqref{eq:localLipschitz}.
The third inequality comes from the triangle inequality and H\"older's inequality,
and the fourth inequality arises directly from Lemma \ref{lem:ensdist} and the boundedness
of $\bar{\hhv}_n$ in $L^p$ for $p\geq 2$.}

For the second summand~{of~\eqref{eq:lpertri}}, notice that we can
write $\bar{\mu}^L_n = \sum_{\ell=0}^L \bar{\mu}^\ell_n - \bar{\mu}^{\ell-1}_n$,
where $\bar{\mu}^\ell_n$ is the measure associated to the level $\ell$ limiting process $\bar{v}^\ell$. 
Then, by virtue of~\eqref{eq:mzin0} and condition (ii) of Assumption~\ref{ass:mlrates},
\begin{equation}\label{eq:finvar}
 \begin{split}
\|\bar{\mu}^{\rm ML}_n (\varphi) - \bar{\mu}^L_n(\varphi)\|_p  & \leq
\sum_{\ell=0}^L \left\|E_{M_\ell}\Big[\varphi(\bar{\hhv}_n^\ell) - \varphi(\bar{\hhv}_n^{\ell-1}) - 
\bbE[\varphi(\bar{\hhv}_n^\ell) - \varphi(\bar{\hhv}_n^{\ell-1})] \Big] \right\|_p  \\
 & \leq \hh{C} \sum_{\ell=0}^L M_\ell^{-1/2}\| \varphi(\bar{\hhv}_n^\ell) - \varphi(\bar{\hhv}_n^{\ell-1})\|_p \\
  & \leq C \sum_{\ell=0}^L M_\ell^{-1/2}\| \bar{\hhv}_n^\ell - \bar{\hhv}_n^{\ell-1}\|_{p} \\
& \lesssim \sum_{\ell=0}^L M_\ell^{-1/2}N_\ell^{-\beta/2} \lesssim \varepsilon.
 \end{split}
 \end{equation}

\hh{
Finally, for the bias term, 
\begin{equation}\label{eq:finbias}
\|\bar{\mu}^L_n(\varphi) - \mu_n (\varphi)\|_p =  |\bar{\mu}^L_n(\varphi) - \mu_n (\varphi)|
= \abs{\E{\varphi(\bar{\hhv}^L_n) - \varphi(\widehat{v}_n)}} \lesssim  \varepsilon,
\end{equation}
where the last inequality follows from Lemma~\ref{lem:boundsBarV} and Assumption~\ref{ass:mlrates} (i).}

Putting together \eqref{eq:finensembles}, \eqref{eq:finvar}, and \eqref{eq:finbias}
in~\eqref{eq:lpertri} yields the sought bound in~\eqref{eq:lperror}.
\end{proof}

Theorem~\ref{thm:main} shows the cost-to-$\varepsilon$ performance of
MLEnKF, and to verify that it generally outperforms EnKF in this
performance measure, we end this section with a comparable
result on the cost-to-$\varepsilon$ perfomance of EnKF.

\begin{theorem}[EnKF accuracy vs. cost]\label{thm:mainEnKF}
\kl{
Suppose Assumption~\ref{ass:psilip}, Assumption~\ref{ass:mlrates} (i),
and  Assumption~\ref{ass:mlrates} (iii) hold. 
For a given $\varepsilon >0$, 
let $L$ and $M$ be defined under the constraints
$L \eqsim \log(\varepsilon^{-1})/\alpha$
and $ M \eqsim  \varepsilon^{-2}$.
Then for all functions $\varphi \in \mathcal{F}$ that are locally Lipschitz
continuous with at most polynomial growth at infinity,
cf.~Definition~\ref{def:localLip}, we have for any $p\ge2$,
\begin{equation}
\|\mu^{\rm MC}_n (\varphi) - \mu_n (\varphi) \|_p \lesssim \varepsilon.
\label{eq:lperrorEnKF}
\end{equation}
Here $\mu^{\rm MC}_n$ denotes the EnKF empirical measure defined in
\eqref{eq:emp}, where the samples are given by the EnKF predict
formulae at resolution level $L$ (i.e., with the numerical integrator
$\Psi^L$), approximating the time $t_n=n$ mean-field EnKF distribution
$\mu_n$. The computational cost of the EnKF estimator over the
time sequence satisfies
\begin{equation}\label{eq:mlenkfCostsEnKF}
\cost{\mathrm{EnKF}} \lesssim \varepsilon^{-(2+\gamma/\alpha)}.
\end{equation}
}
\end{theorem}

\begin{proof}[Sketch of proof]
  \kl{By the triangle inequality
\[
\begin{split}
\|\mu_n(\varphi) -\mu^{\rm MC}_n (\varphi) \|_p & \leq 
\norm{\mu_n(\varphi) -  \bar{\mu}_n^L (\varphi)}_p,
+\norm{\bar{\mu}^L(\varphi) - \bar{\mu}_n^{\rm MC} (\varphi) }_p \\
& +  \norm{\bar{\mu}_n^{\rm MC} (\varphi) -\mu^{\rm MC}_n (\varphi)}_p
 \eqcolon I + II + III,
\end{split}
\]
where $\bar{\mu}_n^{\rm MC}$ denotes the empricial measure associated
to an EnKF ensemble $\{\bar{\hhv}^L_n(\omega_i) \}_{i=1}^M$ and
$\bar{\mu}^{L}_n$ denotes the emprical measure associated to
$\bar{\hhv}^L_n$. We bound the terms $I, II$, and $III$ individually.

For the first term, we have 
\[
I = \abs{\mu_n(\varphi) -  \bar{\mu}_n^L (\varphi)} \leq
\abs{\E{\varphi(\widehat v_n) - \varphi(\bar{\hhv}^L_n) }} \lesssim \varepsilon,
\]
where the last inequality is implied by 
inequality~\eqref{eq:barVWeak} of Lemma~\ref{lem:boundsBarV},
which it is straightforward to verify holds under
Assumption~\ref{ass:mlrates} (i). 

For the second term, we first note that 
we may assume without loss of generality that 
$\varphi(0) = 0$. Since $\varphi$ is locally Lipschitz
continuous with at most polynomial growth at infinity,
there then exists positive scalars $\nu, C_\varphi$ such that
\[
|\varphi(x)| \leq C_\varphi (1 + |x|^\nu).
\] 
By inequality~\eqref{eq:mzin0}  and since $\bar{v}_n^L \in
L^p(\Omega)$ for any $p>1$, 
\begin{equation*}
II \leq \norm{E_{M}[ \varphi(\bar{\hhv}^L_n)] - \E{\varphi(\bar{\hhv}^L_n)} }_p  \leq M^{-1/2}
\norm{\varphi( \bar{\hhv}^L_n) }_p 
\leq  C \varepsilon \norm{ 1 + \abs{\bar{\hhv}^L_n}^\nu }_{p} \lesssim \varepsilon.
\end{equation*}

For the last term, let us first assume that for any $p\ge 2$ and 
finite $n$, 
\begin{equation}\label{eq:termIIIAssumption}
\norm{\hhv_n^L -  \bar{\hhv}_n^L } \lesssim \varepsilon,
\end{equation}
for the particle dynamics $\hhv_n^L$ and $\bar{\hhv}_n^L$ respectively
 associated to the EnKF ensemble $\{\hhv_{n,i}^L \}_{i=1}^M$ and 
the mean-field EnKF ensemble $\{\bar{\hhv}_{n,i}^L\}_{i=1}^M$.
Then the assumed regularity of $\varphi$, 
that $\hhv_n^L, \bar{\hhv}_n^L \in L^p(\Omega)$ for all $p\ge 2$,
and H\"older's inequality yield that
\[
\begin{split}
III &= \norm{ E_{M}[ \varphi(\hhv_n^L) - \varphi(\bar{\hhv}_n^L)] }
\leq C_\varphi \norm{\abs{\hhv_n^L - \bar{\hhv}_n^L} 
\Big(1 +\abs{\hhv_n^L}^\nu + \abs{\bar{\hhv}_n^L}^\nu\Big) }_p \\
& \lesssim \norm{\hhv_n^L - \bar{\hhv}_n^L }_p \lesssim \varepsilon.
\end{split}
\]
All that remains is to verify~\eqref{eq:termIIIAssumption}.  Since this can
be done by very similar steps as in the proof of
inequality~\eqref{eq:ensdist}, we omit this verification. 

}

\end{proof}

\hh{
\begin{remark}
Notice that for a given $n$ one can 
obtain an error $\cO(\varepsilon)$ for MLEnKF
in \eqref{eq:lperror} for an additional cost which is given by replacing 
$\varepsilon$ by $\varepsilon|\log \varepsilon|^{-n}$
in \eqref{eq:mlenkfCosts2}.
Furthermore, it is worth noting that, for any $m$ and for any $\delta>0$,
$|\log\varepsilon|^m=\cO(\varepsilon^{-\delta})$.  
Hence one can obtain a cost-of-error rate in \eqref{eq:mlenkfCosts2}
which is uniform in time and asymptotically superior to 
EnKF \eqref{eq:mlenkfCostsEnKF}.

\end{remark}
}



\section{Numerical Examples} 
\label{sec:numerics}
  
In this section the performance of EnKF and MLEnKF are compared
on some very simple numerical examples in terms of computational cost vs.~approximation 
error. First, in section \ref{ssec:ou}, underlying dynamics from an Ornstein--Uhlenbeck 
SDE is considered. Next, in section~\ref{ssec:gbm},
the underlying dynamics geometric Brownian motion is considered. 
Both of these examples are indeed analytically tractable, however they 
are approximated as though they were not. This provides a
solid benchmark to compute errors and allows the theory to be illustrated.

\subsection{An Ornstein-Uhlenbeck SDE}
\label{ssec:ou}
We first consider the simple Ornstein--Uhlenbeck SDE problem 
\begin{equation}\label{eq:OuSDE}
 du = -u dt + \sigma dW_t, \qquad u(0) = 1.
\end{equation}
It has the exact solution
\[
u(t) = u(0)e^{-t}  + \int_0^t \sigma e^{(s-t)} dW_s,
\]
and since 
\[
  \int_0^{1} \sigma e^{(s-1)} dW_s \sim N\Big(0, \underbrace{ \frac{\sigma^2}{2}(1 - e^{-2 } )}_{=:\Sigma} \Big),
\]
one SDE realization sampled at the observation times $t_n = n$ is generated by 
the linear solution operator
\[
u_{n+1} = e^{-1} u_n  +  \xi_n =: \Psi(u_n)
\]
where $\xi_n \sim N( 0, \Sigma)$ i.i.d. 
The corresponding noisy observations are given by
\[
{y_n^\obs} = u_n + \eta_n,
\]
with $\eta_n \sim N(0,\Gamma)$ i.i.d.

For the MLEnKF algorithm, a hierarchy of Milstein 
solution operators $\{\psiL\}_{\ell=0}^\infty$ are introduced,
where the $\ell^{\text{th}}$
level solution operator uses a uniform time-step of size 
$\Dt{\ell} = 2^{-(\ell+1)}$. A numerical integration step takes the form 
\begin{equation}\label{eq:eulerMaruyama}
u^{\ell}_{n,m+1}= u^{\ell}_{n,m} (1-\Dt{\ell}) + \sigma \DW{\ell}_{n,m}, \qquad m =0,1,\ldots, 2^{\ell+1} -1, 
\end{equation}
where the initial condition is given by $v^{\ell}_{n,0} =\widehat{v}^{\ell}_{n-1}$,
\[
\DW{\ell}_{n,m} = W(t_n+(m+1)\Dt{\ell}) - W(t_n+m\Dt{\ell}) \sim N(0,\Dt{\ell}),
\]
and $u^{\ell}_{n} = u^{\ell}_{n, 2^{\ell+1}}$.

Moreover, since the solution operator for~\eqref{eq:OuSDE} is linear, the gold standard becomes the conventional Kalman filter update
\[
(\widehat{m}^\dagger_n, \widehat{C}^\dagger_n) = {\Big( (I-K_{n}H)\meanHat{n} +
K_{n} y_{n}^\obs, (I-K_{n}H) \covHat{n} \Big )}, \quad n=1,2, \ldots.
\]

\subsubsection*{Problem parameters}
In the numerical experiments, \hh{$N=100, 200$ and $400$} observation times
$\{t_n=n\}_{n=1}^N$ are used, and the covariance \kl{parameters are set to
$\Gamma=0.04$} and $\sigma =0.5$. For a prescribed computational cost $\bigO{J}$, an EnKF
ensemble of size $M= \bigO{J^{2/3}}$ is solved by the 
\hh{Milstein} method on a mesh $\Delta t = \bigO{J^{1/3}}$,
and for the MLEnKF method, we set $L$ and $M_\ell$ according to the constraint in
Theorem~\ref{thm:main}.

\subsubsection*{Approximations of the mean and covariance}

In our first numerical experiment we approximate the gold standard mean and covariance 
\kl{for a single observation realization}
using the 
respective ensemble Kalman filtering methods, and measure the approximation error in terms of 
the root mean square error (RMSE):
\hh{\begin{equation}\label{eq:errorMeasure}
\sqrt{ \sum_{n=1}^N \frac{|\widehat{m}^\dagger_n - \widetilde{m}_n|^2}{N} }, \qquad 
\sqrt{\sum_{n=1}^N \frac{|\widehat{C}^\dagger_n - \widetilde{C}_n|^2}{N}},
\end{equation}}
with $(\widetilde{m}_n, \widetilde{C}_n)$ denoting \kl{a single
  realization of} either the EnKF or the MLEnKF updates approximating
the gold standard moments. These observables are sufficiently smooth
to reach the rates $\alpha=1$ and $\beta = 2$ with the Milstein
method, cf.~\cite{GrahamTalay}. \hhLast{The respective decay rates are numerically 
verified over a sequence of times in Figure~\ref{fig:levelsEx1}.} Figure~\ref{fig:fig1Ex1} presents a
numerical performance study measuring RMSE~\eqref{eq:errorMeasure}
vs.~computational cost for the respective methods.  As is to be
expected from Theorem~\ref{thm:main} the decay of RMSE for the MLEnKF
method as a function of the cost $J$ is roughly $\cO(J^{-1/2})$,
orders of magnitude faster than the observed and expected EnKF decay
rate $\cO(J^{-1/3})$. \hh{
Note that the error growth
  factor $\abs{\log(\varepsilon)}^n$ from the theoretical
  bound~\eqref{eq:lperror} is not visible in the experiments.  In fact, 
  the constant is even stable (the shift in cost as measured by runtime is simply due to 
  computation of additional updates), indicating that with a more careful analysis 
  the present results may be extended to an infinite time horizon. }

\begin{figure}[htbp]
  \centering
  \includegraphics[width=0.95\textwidth]{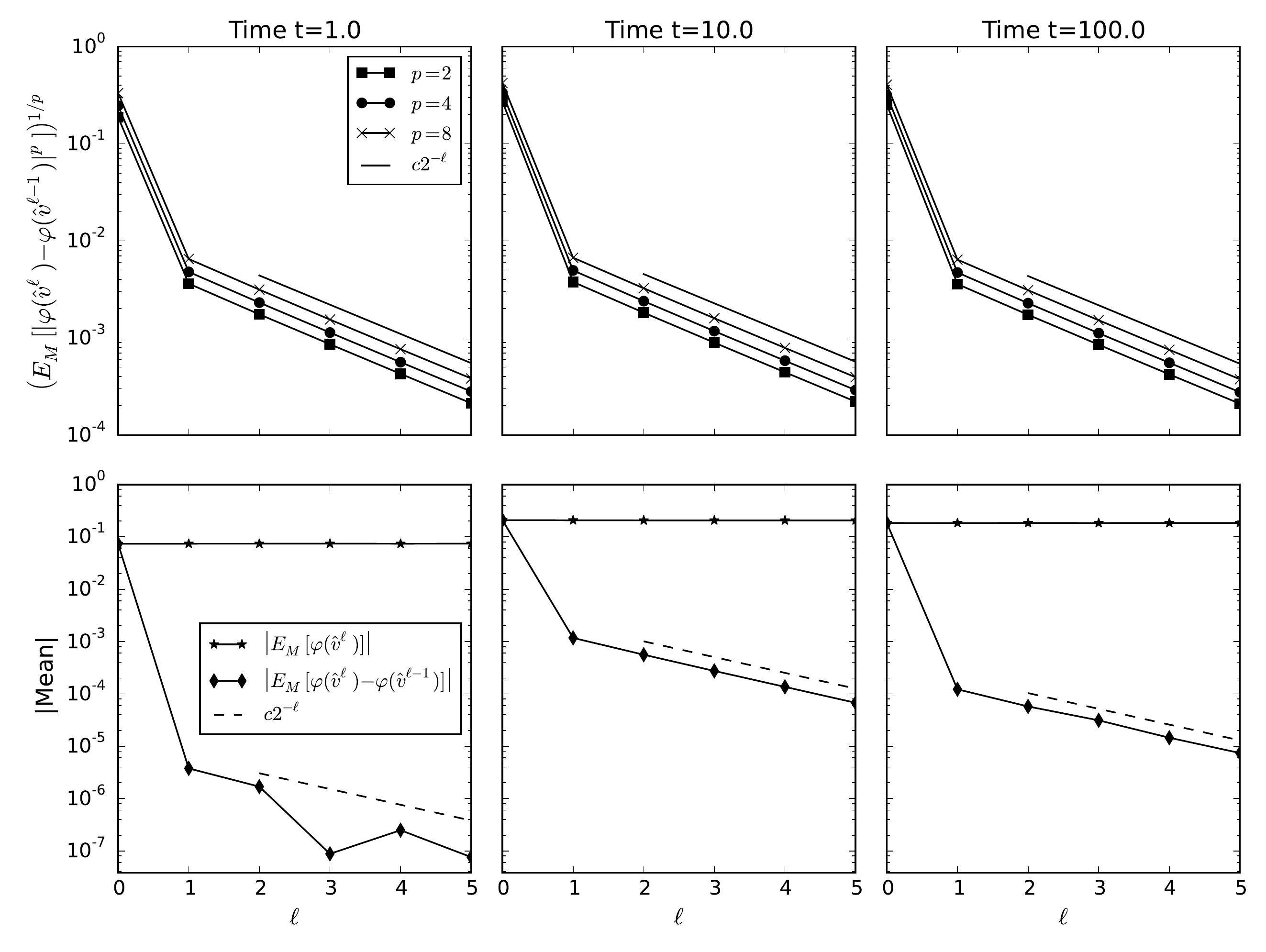}
   \caption{Numerical estimates of the decay rates over a sequence of times 
     for the problem presented in Section~\ref{ssec:ou} with $\varphi(v) = v$.
     The computations use $M=10^6$ particles on every level.}
  \label{fig:levelsEx1}
\end{figure}

\begin{figure}[htbp]
  \centering
  \includegraphics[width=1\textwidth]{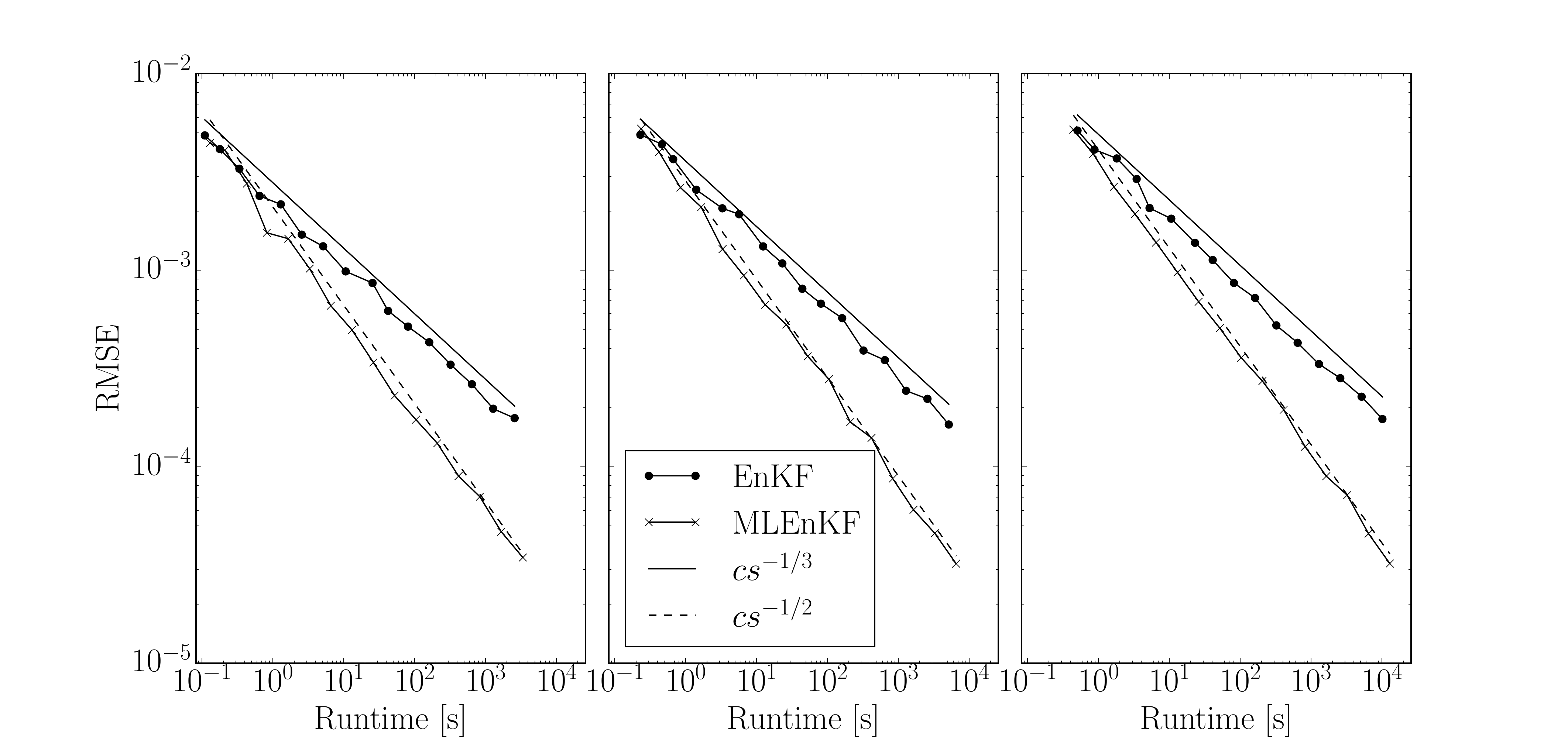}
   \includegraphics[width=1\textwidth]{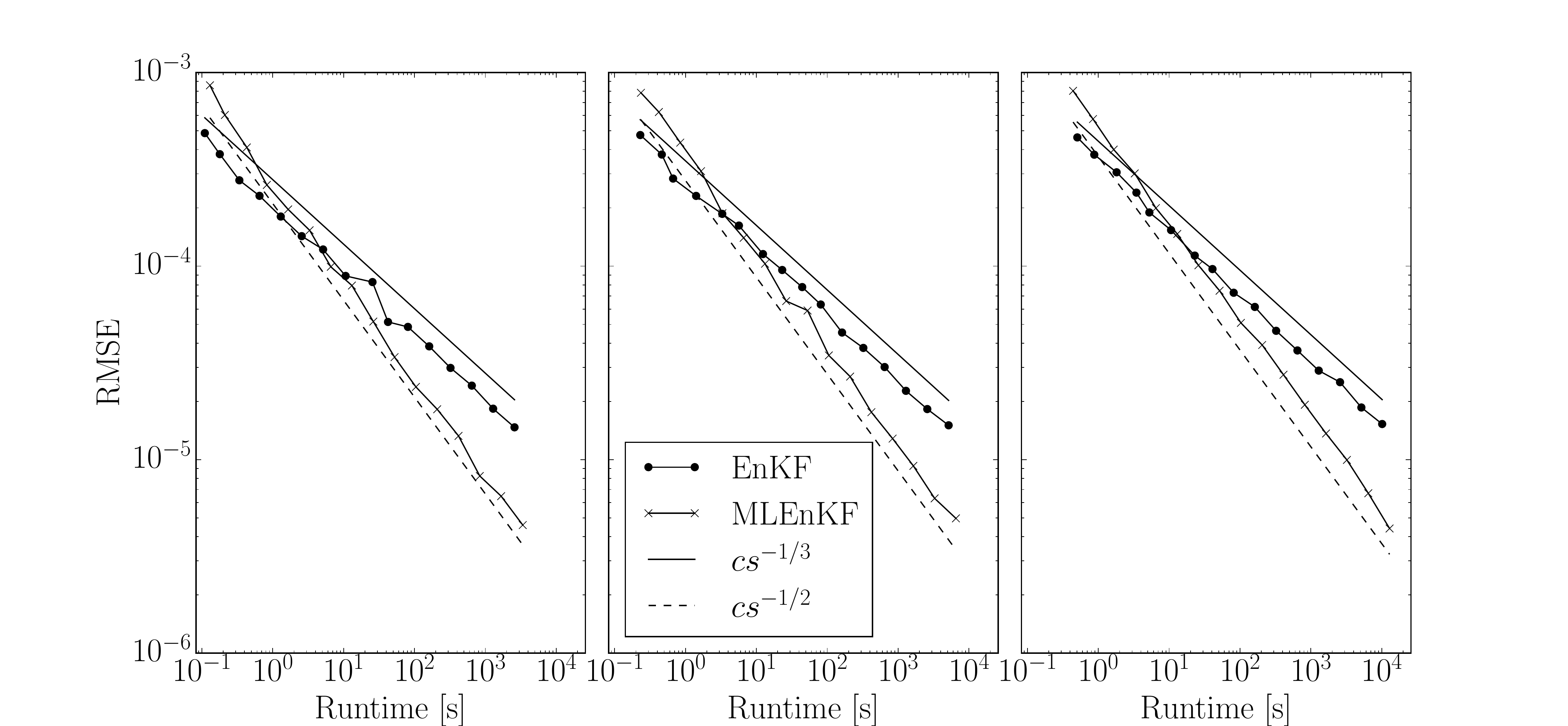}
  \caption{Comparison of the accuracy vs.~computational cost when using
  the EnKF and MLEnKF methods on the filtering problem presented in Section~\ref{ssec:ou}.
    \hh{The error is measured in terms of the RMSE~\eqref{eq:errorMeasure} for the 
  mean (top row) and covariance (bottom row), computed with $N=100, 200$ and
  $400$ observation times in the first, second and third column, respectively.
  The computational cost is measured in computer runtime.}}
  \label{fig:fig1Ex1}
\end{figure}

\subsubsection*{Approximations of the excedence probability}
In our second numerical test, we approximate the mean of the
observable $\varphi(\widehat{u}_n) :=
\mathbf{1}\{\widehat{u}_n>0.1\}$, which corresponds to the excedence
probability $\bbP(\widehat{u}_n>0.1) =
1-\Phi((0.1-\widehat{m}_n^\dagger)/\sqrt{\widehat{C}_n^\dagger})$.
The Milstein method achieves the weak rate $\alpha =1$, but while one
may show for $p=2$ and any $\delta >0$, $\| \varphi(\psiL(v)) -
\varphi(\Psi^{\ell-1}(v)) \|_p \lesssim N_\ell^{(1-\delta)/2}$,
cf.~\cite{GilesMCQMC06, Avikainen09}, the low regularity of the
observable implies that there does not exist a $\beta >0$ fulfilling
condition (ii) of Assumption~\ref{ass:mlrates} for all $p >
2$. \hhLast{A numerical inference of $\beta=0$ can be made from the
  numerical estimates of the decay rates in
  Figure~\ref{fig:levelsExProb}, where we see that the decay rate of
  $\|\varphi(\psiL(v)) - \varphi(\Psi^{\ell-1}(v))\|_p$ consistently
  decreases towards $0$ as $p$ increases over a sequence of times
  (while $\alpha \approx 1$).} Theorem~\ref{thm:main} does therefore
not cover the given approximation problem. Nonetheless, implementing
with the rates $\beta =1$ and $\alpha=1$, a numerical comparison of
the performance of EnKF and MLEnKF approximating the excedence
probability is presented in Figure~\ref{fig:fig1Ex2}. A near optimal
RMSE decay rate, slightly slower than $\cO(J^{-1/2})$, is again
achieved for the MLEnKF method.

\begin{figure}[htbp]
  \centering
  \includegraphics[width=0.95\textwidth]{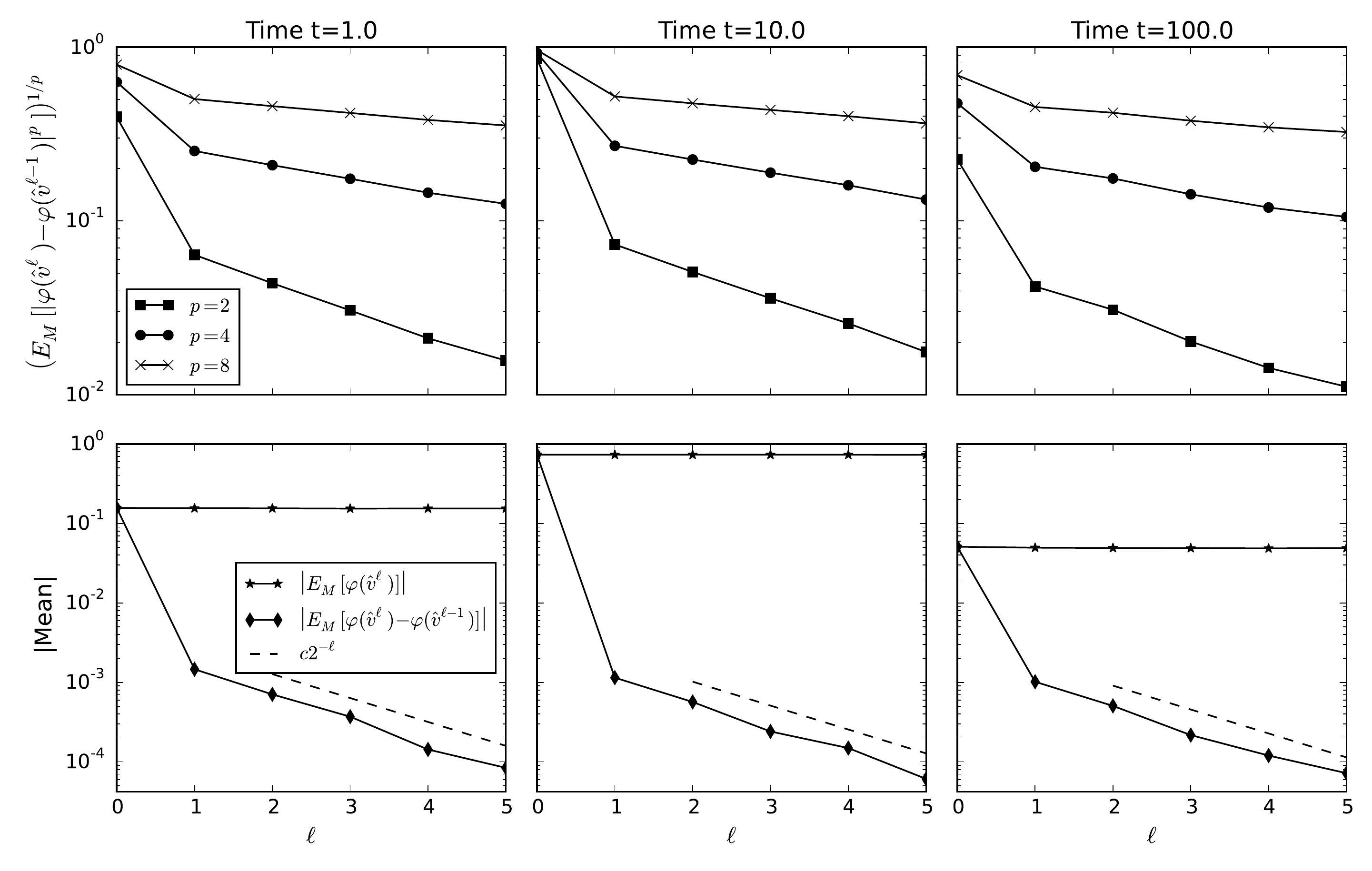}
  \caption{Numerical estimates of the decay rates over a sequence of
    times for the problem presented in Section~\ref{ssec:ou} when
    approximating the excedence probability $\bbP(\widehat{u}_n>0)$,
    i.e., with $\varphi(v) = \mathbf{1}\{v>0.1\}$.
    The computations use $M=10^6$ particles on every level.}
  \label{fig:levelsExProb}
\end{figure}

\begin{figure}[htbp]
  \centering
 \includegraphics[width=0.6\textwidth]{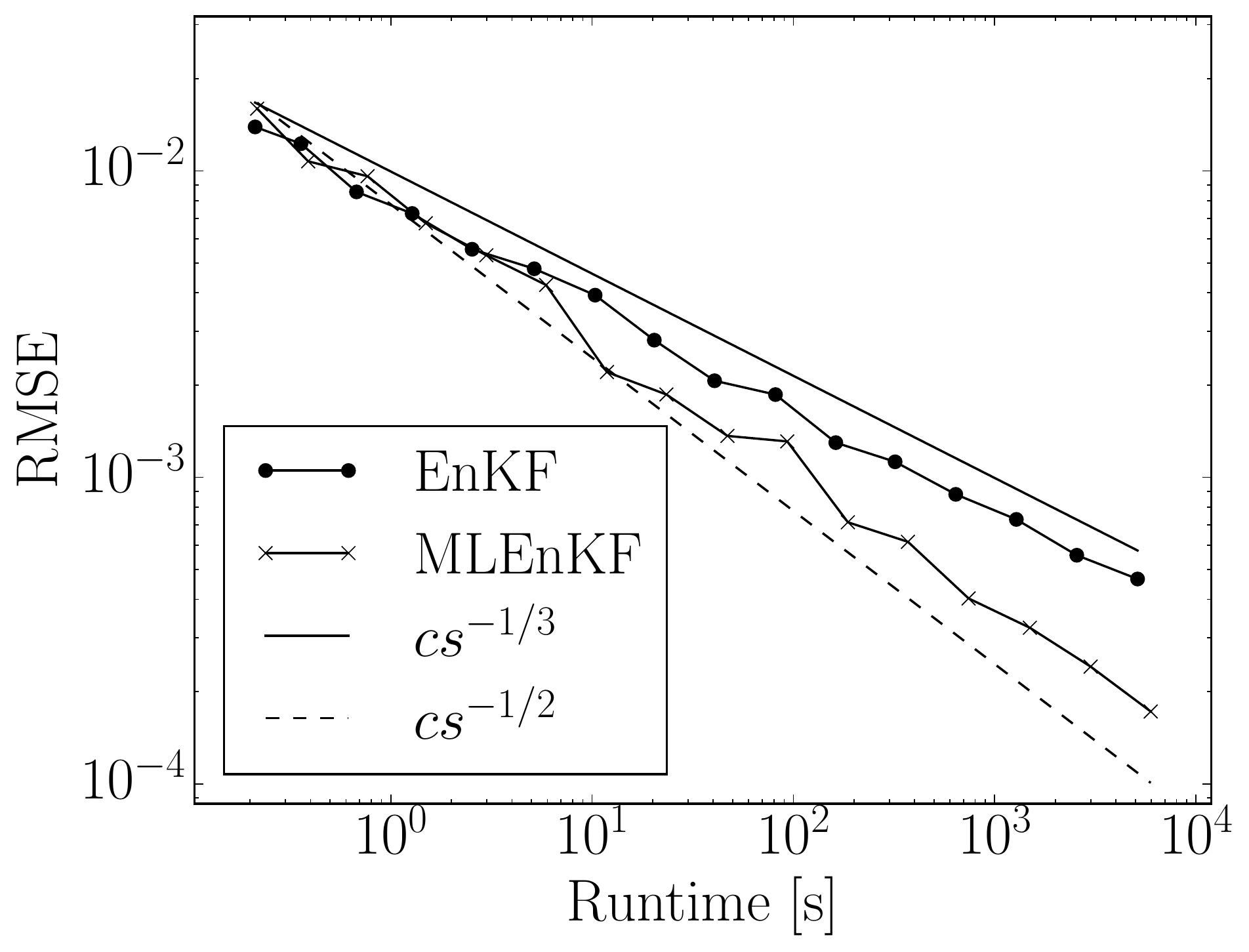}
  \caption{Accuracy vs.~computational cost comparison of the EnKF and
    MLEnKF methods on the filtering problem presented in
    Section~\ref{ssec:ou} when approximating the excedence probability
    $\bbP(\widehat{u}_n>0.1)$ \kl{over $N=200$ observation times.}  The
    error is measured in terms of the RMSE and the computational cost
    in computer runtime.}
  \label{fig:fig2Ex1}
\end{figure}

\subsection{Drift-alternating Geometric Brownian Motion}
\label{ssec:gbm}

We next consider the SDE
\begin{equation}\label{eq:gbm}
du(t+n) = \begin{cases} \sigma^2 u(t+n) dt + \sigma u(t+n)
  dW(t+n),& \text{if } n \text{ is even},\\
 \sigma u(t+n) dW(t+n), & \text{else,}
\end{cases}
\quad \text{for } t \in (0,1),
\end{equation}
and with the initial condition $u(0) = 1$.
This equation is analytically tractable as well, and the solution
of the transformed equation $z = \log u$ is given via It{\^o}'s formula
by
\[
dz(t+n) = (-1)^n\frac{\sigma^2}{2} dt + \sigma dW(t+n). 
\]
Defining $\xi_n \sim N(0,\sigma^2)$ i.i.d., one has that 
\[
z_{n+1} = z_n +  (-1)^n\frac{\sigma^2}{2} + \xi_n =: \Psi_n(z_n), \quad {\rm with} \quad  z_0 = \log u_0 = 0, 
\]
and the solution of~\eqref{eq:gbm} can be obtained via exponentiation: $u_n = e^{z_n}$.
Moreover, noisy observations for $u_n$ are introduced on the form  
\[
\tilde{y}_n = u_n e^{\eta_n},
\]
and $\eta_n \sim N(0,\Gamma)$ i.i.d.
Which, upon defining ${y_n^\obs} = \log \tilde{y}_n$, yields the following relation to 
noisy observations of $z_n$:
\[
{y_n^\obs} = z_n + \eta_n.
\]
As the SDE~\eqref{eq:gbm} does not fulfill the linear Gaussian constraints~\eqref{eq:linnonaut}
but $z = \log u$ does, we will here update the ensemble of $z = \log
u$ processes. However, to add some artificial difficulty to the problem, the 
numerical integration is done on the $u$ ensemble:   
\begin{enumerate}

 \item[(i)] Numerically integrate a (multilevel or single level) ensemble $u_{n-1} \to u_{n}$.
 
 \item[(ii)] Compute sample mean and covariance of ${z_n|Y_{n-1}^\obs}$ using the $z_{n} = \log u_n$ ensemble.
 
 \item[(iii)] Update the ensemble $z_n$ by the new information provided by the observation ${y_n^\obs}$.
 
 \item[(iv)] Compute the initial condition for the ensemble $u_n = e^{z_n}$ and return to (i).

\end{enumerate}

\begin{remark}
The numerical integration of the GBM process in step (i) above
introduces an artificial difficulty in the filtering problem
since the integration may by other means be solved exactly. In practice, this does of
course not make sense, but our purpose here is simply to numerically
validate the performance of the MLEnKF method on a set of 
simple filtering problems for which reference solutions exist.
\end{remark}

Numerical integration of $u_n$ is done by the hierarchy 
of Euler--Maruyama schemes introduced in~\eqref{eq:eulerMaruyama}
(applied to the GBM problem, the schemes are Euler--Maruayama, while 
applied to problems with additive noise, the schemes are Milstein), 
here with the slightly finer mesh hierarchy $\Dt{\ell} =
2^{-3+\ell}$, since the problem less stable. The covariance parameters
are set to $\sigma = 1/4$ and \kl{$\Gamma = 1/16$, $N=200$}
and the numerical method yields the rates $\alpha=1$, $\beta =1$ (and $\gamma=1$).
\hhLast{See Figure~\ref{fig:altLevels} for a numerical verfication of these decay rates over 
a sequence of times.}
In Figure~\ref{fig:fig1Ex2}, the gold standard mean and covariance of
$z_n$ has been approximated by the filtering methods. We observe
an RMSE decay rate slightly slower than 
$\cO(s^{-1/2})$ for MLEnKF and $\cO(s^{-1/3})$ for EnKF, where $s$ denotes runtime in seconds.

\begin{figure}[htbp]
  \centering
  \includegraphics[width=.95\textwidth]{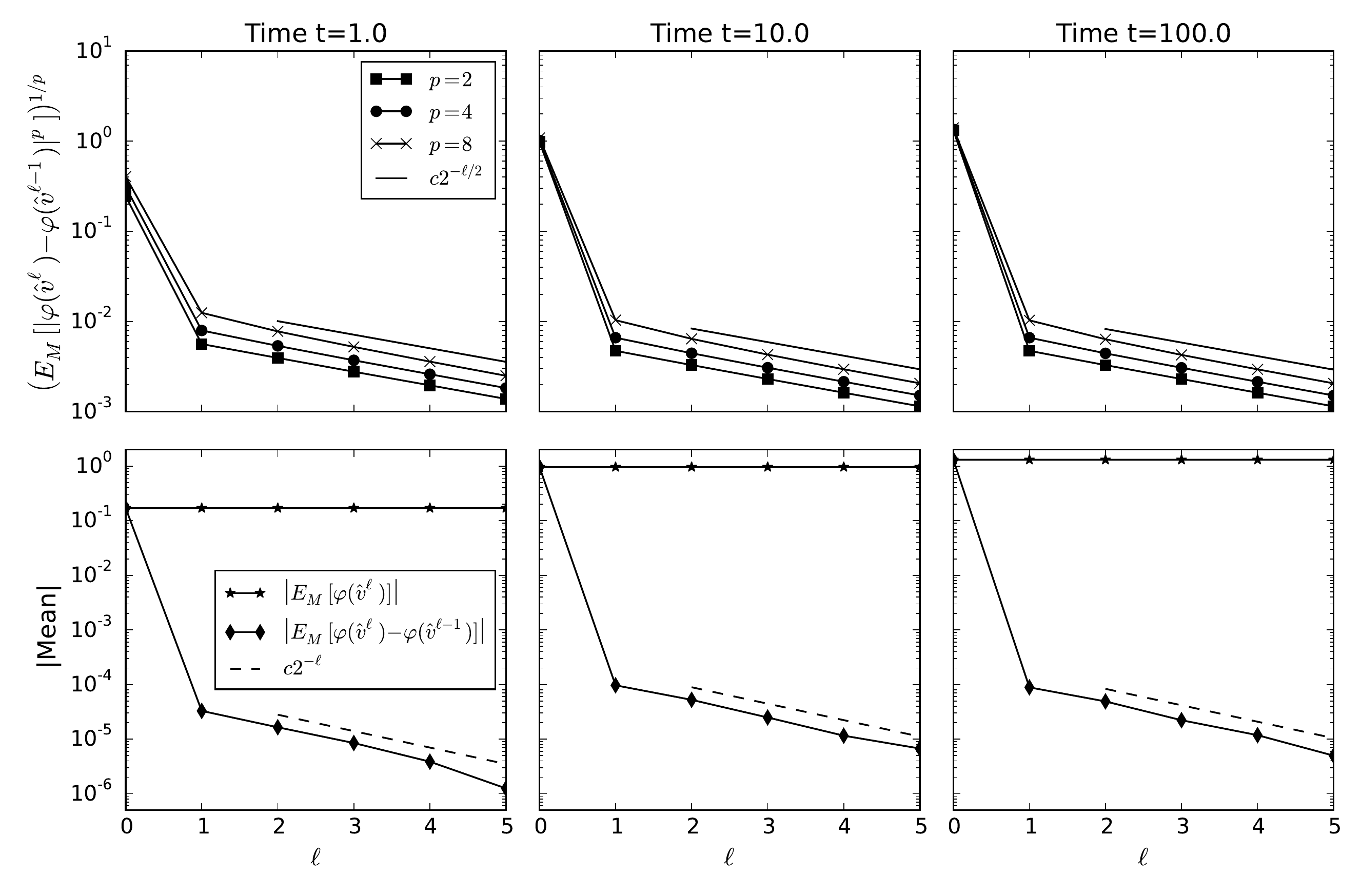}
  \caption{Numerical estimates of the decay rates over a sequence of
    times for the problem presented in Section~\ref{ssec:gbm} with
    $\varphi(v) = v$.  The computations use $M=10^6$ particles on
    every level.}
  \label{fig:altLevels}
\end{figure}

\begin{figure}[htbp]
  \centering
\includegraphics[width=0.47\textwidth]{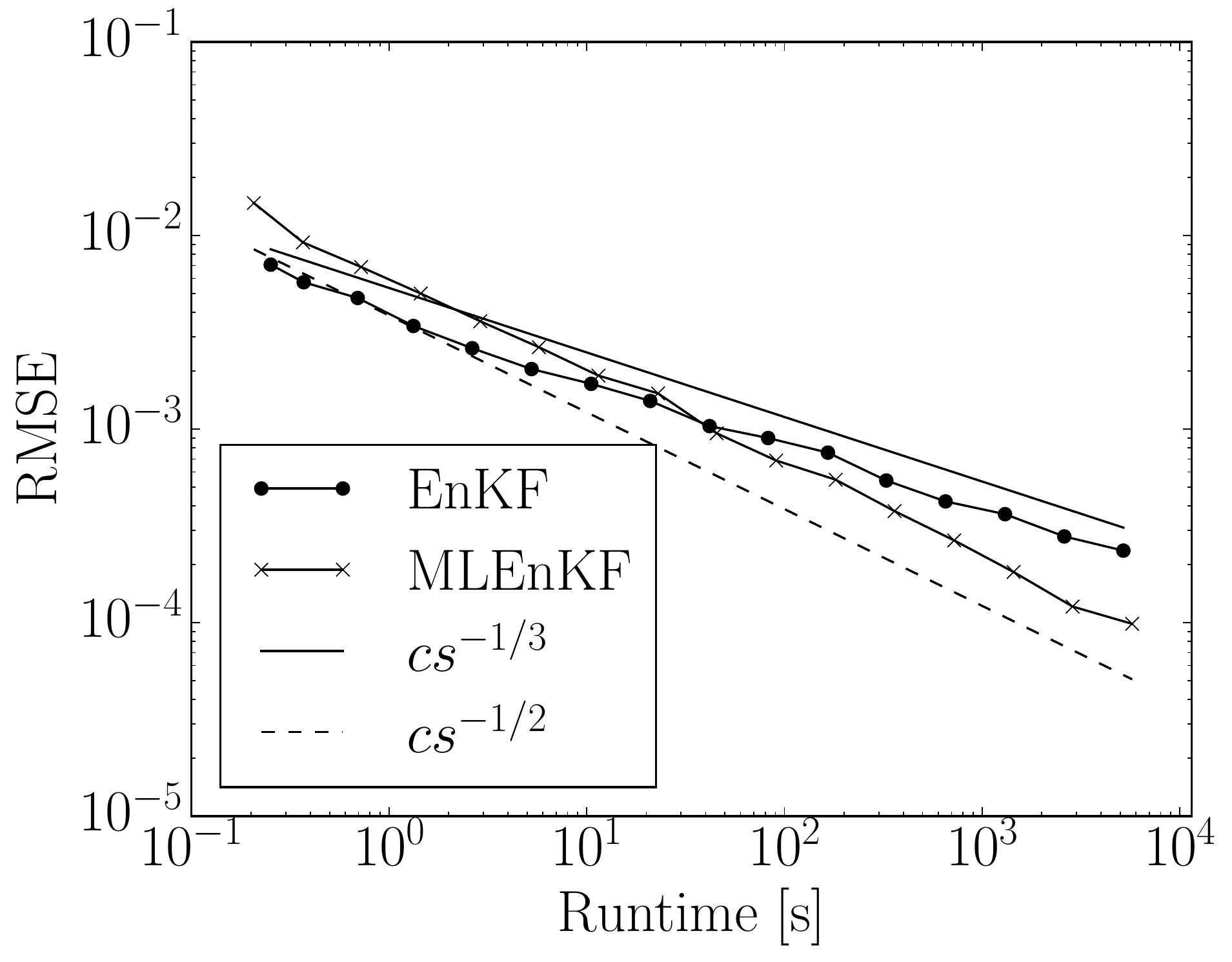}
\includegraphics[width=0.47\textwidth]{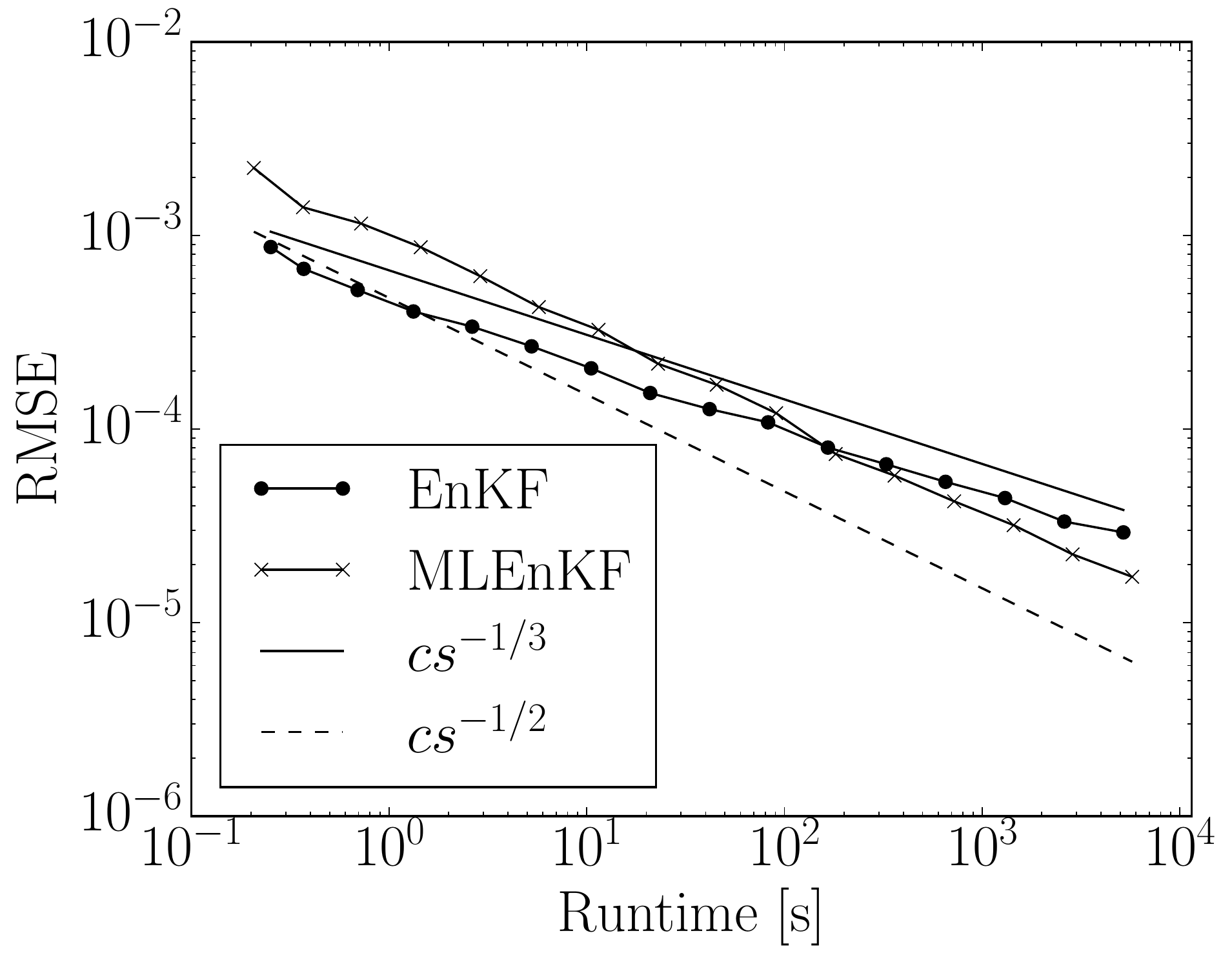}
  \caption{Accuracy vs.~computational cost comparison of
  the EnKF and MLEnKF methods on the filtering problem presented in Section~\ref{ssec:gbm}.
  The error is measured in terms of the RMSE~\eqref{eq:errorMeasure} for the 
  mean (left plot) and covariance (right plot),
  and the computational cost is measured in computer runtime.}
  \label{fig:fig1Ex2}
\end{figure}

\section{Conclusion}
\label{sec:conclusion}

A first attempt, to the knowledge of the authors, at filtering using a multilevel Monte Carlo approach 
is considered in the present work.  A proof based on induction of the optimality of the cost as a function of the error, or 
equivalently the error as a function of the  cost, is given.  
This shows that an optimality result, which is slightly penalized with respect to the vanilla Monte Carlo result, can extend to the case of sequential inference.
There is either a logarithmic term which grows with the number of steps, or a slightly higher rate $\cO(\varepsilon^{-2-\delta})$, for $\delta>0$. 
The ensemble Kalman filter is considered, which is consistent only in the case of a linear Gaussian model.  
However, the mean-field limiting equation may be viewed as a one-step optimal linear 
(in the observation) filter, and the convergence to this limiting distribution has the desired rate for a more general class of models.  
\rev{Since this work was submitted, two papers have appeared which consider consistent nonlinear filters for similar models.  
The paper \cite{jasra2015multilevel} considered multilevel particle filters with optimally coupled multinomial resampling 
and found the rate of strong convergence 
$\beta$ is effectively reduced by a factor of 2 as a result of the resampling.  
The paper \cite{gregory2015multilevel} considered multilevel ensemble transform particle filters, which 
use an optimally coupled deterministic transformation in place of the standard random resampling mechanism, 
and numerical results indicated the rate may be reduced in some cases and the same in others.}

{\bf Acknowledgements } Research reported in this publication was supported by the King Abdullah University of Science and Technology (KAUST).  HH, KJHL, and RT were members of the SRI Center for Uncertainty Quantification at KAUST for much of the research reported.  
KJHL was additionally supported by an ORNL LDRD Strategic Hire grant.

\bibliography{mybib}

\def\cprime{$'$} \def\cprime{$'$} \def\cprime{$'$} \def\cprime{$'$}
  \def\cprime{$'$} \def\cprime{$'$} \def\cprime{$'$}
  \def\Rom#1{\uppercase\expandafter{\romannumeral #1}}\def\u#1{{\accent"15
  #1}}\def\Rom#1{\uppercase\expandafter{\romannumeral #1}}\def\u#1{{\accent"15
  #1}}\def\cprime{$'$} \def\cprime{$'$} \def\cprime{$'$} \def\cprime{$'$}
  \def\cprime{$'$} \def\cprime{$'$}
\begin{thebibliography}{10}

\bibitem{anderson2012localization}
{\sc Jeffrey~L Anderson}, {\em Localization and sampling error correction in
  ensemble kalman filter data assimilation}, Monthly Weather Review, 140
  (2012), pp.~2359--2371.

\bibitem{Avikainen09}
{\sc Rainer Avikainen}, {\em On irregular functionals of {SDE}s and the {E}uler
  scheme}, Finance Stoch., 13 (2009), pp.~381--401.

\bibitem{BC09}
{\sc A.~Bain and D.~Crisan}, {\em Fundamentals of Stochastic Filtering},
  Springer, 2009.

\bibitem{bickel2008covariance}
{\sc Peter~J Bickel and Elizaveta Levina}, {\em Covariance regularization by
  thresholding}, The Annals of Statistics,  (2008), pp.~2577--2604.

\bibitem{bickel2008regularized}
\leavevmode\vrule height 2pt depth -1.6pt width 23pt, {\em Regularized
  estimation of large covariance matrices}, The Annals of Statistics,  (2008),
  pp.~199--227.

\bibitem{Chernov14}
{\sc Claudio Bierig and Alexey Chernov}, {\em {Convergence analysis of
  multilevel Monte Carlo variance estimators and application for random
  obstacle problems}}, Numerische Mathematik,  (2014), pp.~1--35.

\bibitem{burgers1998analysis}
{\sc Gerrit Burgers, Peter Jan~van Leeuwen, and Geir Evensen}, {\em Analysis
  scheme in the ensemble {K}alman filter}, Monthly weather review, 126 (1998),
  pp.~1719--1724.

\bibitem{cappe2005inference}
{\sc Olivier Capp{\'e}, Eric Moulines, and Tobias Ryd{\'e}n}, {\em Inference in
  hidden Markov models}, Springer, 2005.

\bibitem{Carlsson10}
{\sc J.~Carlsson, Moon K.S., Szepessy A., Tempone R., and Zouraris G.}, {\em
  Stochastic differential equations: Models and numerics}.
\newblock Lecture notes, 2010.

\bibitem{Cliffe11}
{\sc K.~A. Cliffe, M.~B. Giles, R.~Scheichl, and A.~L. Teckentrup}, {\em
  Multilevel {M}onte {C}arlo methods and applications to elliptic {PDE}s with
  random coefficients}, Comput. Vis. Sci., 14 (2011), pp.~3--15.

\bibitem{CollierBIT}
{\sc Nathan Collier, Abdul-Lateef Haji-Ali, Fabio Nobile, Erik von Schwerin,
  and Raúl Tempone}, {\em {A continuation multilevel Monte Carlo algorithm}},
  BIT Numerical Mathematics,  (2014), pp.~1--34.

\bibitem{del2004feynman}
{\sc Pierre Del~Moral}, {\em Feynman-Kac Formulae: Genealogical and Interacting
  Particle Systems with Applications}, Springer, 2004.

\bibitem{doucet2000sequential}
{\sc Arnaud Doucet, Simon Godsill, and Christophe Andrieu}, {\em On sequential
  {Monte Carlo sampling methods for Bayesian filtering}}, Statistics and
  computing, 10 (2000), pp.~197--208.

\bibitem{evensen1994sequential}
{\sc Geir Evensen}, {\em Sequential data assimilation with a nonlinear
  quasi-geostrophic model using {Monte Carlo} methods to forecast error
  statistics}, Journal of Geophysical Research: Oceans (1978--2012), 99 (1994),
  pp.~10143--10162.

\bibitem{evensen2003ensemble}
\leavevmode\vrule height 2pt depth -1.6pt width 23pt, {\em The ensemble
  {K}alman filter: Theoretical formulation and practical implementation}, Ocean
  dynamics, 53 (2003), pp.~343--367.

\bibitem{fedorenko1961relaxation}
{\sc Radii~Petrovich Fedorenko}, {\em A relaxation method for solving elliptic
  difference equations}, Zhurnal Vychislitel'noi Matematiki i Matematicheskoi
  Fiziki, 1 (1961), pp.~922--927.

\bibitem{GilesMCQMC06}
{\sc Mike Giles}, {\em Improved multilevel {Monte C}arlo convergence using the
  {M}ilstein scheme}, in Monte Carlo and Quasi-Monte Carlo Methods 2006,
  Alexander Keller, Stefan Heinrich, and Harald Niederreiter, eds., Springer
  Berlin Heidelberg, 2008, pp.~343--358.

\bibitem{Giles08}
{\sc M.~B. Giles}, {\em Multilevel {M}onte {C}arlo path simulation}, Oper.
  Res., 56 (2008), pp.~607--617.

\bibitem{Giles14}
{\sc M.~B. Giles and L.~Szpruch}, {\em Antithetic multilevel {M}onte {C}arlo
  estimation for multi-dimensional {SDE}s without {L}\'evy area simulation},
  Ann. Appl. Probab., 24 (2014), pp.~1585--1620.

\bibitem{GrahamTalay}
{\sc Carl Graham and Denis Talay}, {\em Stochastic simulation and {M}onte
  {C}arlo methods}, vol.~68 of Stochastic Modelling and Applied Probability,
  Springer, Heidelberg, 2013.
\newblock Mathematical foundations of stochastic simulation.

\bibitem{gregory2015multilevel}
{\sc Alastair Gregory, Colin Cotter, and Sebastian Reich}, {\em Multilevel
  ensemble transform particle filtering}, arXiv preprint arXiv:1509.00325,
  (2015).

\bibitem{gut2005probability}
{\sc Allan Gut}, {\em Probability: a graduate course}, vol.~200, Springer,
  2005.

\bibitem{hackbusch1985multi}
{\sc Wolfgang Hackbusch}, {\em Multi-grid methods and applications}, vol.~4,
  Springer-Verlag Berlin, 1985.

\bibitem{heinrich2001multilevel}
{\sc Stefan Heinrich}, {\em Multilevel {Monte C}arlo methods}, in Large-scale
  scientific computing, Springer, 2001, pp.~58--67.

\bibitem{hoang2013complexity}
{\sc Viet~Ha Hoang, Christoph Schwab, and Andrew~M Stuart}, {\em Complexity
  analysis of accelerated {MCMC methods for Bayesian inversion}}, Inverse
  Problems, 29 (2013), p.~085010.

\bibitem{Hoel14}
{\sc H{\aa}kon Hoel, Erik von Schwerin, Anders Szepessy, and Ra{\'u}l Tempone},
  {\em Implementation and analysis of an adaptive multilevel {M}onte {C}arlo
  algorithm}, Monte Carlo Methods Appl., 20 (2014), pp.~1--41.

\bibitem{jasra2015multilevel}
{\sc Ajay Jasra, Kengo Kamatani, Kody~JH Law, and Yan Zhou}, {\em Multilevel
  particle filter}, arXiv preprint arXiv:1510.04977,  (2015).

\bibitem{jaz70}
{\sc A.H. Jazwinski}, {\em Stochastic processes and filtering theory}, vol.~63,
  Academic Pr, 1970.

\bibitem{kalman1960new}
{\sc Rudolph~Emil Kalman et~al.}, {\em A new approach to linear filtering and
  prediction problems}, Journal of basic Engineering, 82 (1960), pp.~35--45.

\bibitem{kal03}
{\sc E.~Kalnay}, {\em Atmospheric Modeling, Data Assimilation and
  Predictability}, Cambridge, 2003.

\bibitem{ketelsen2013hierarchical}
{\sc C~Ketelsen, R~Scheichl, and AL~Teckentrup}, {\em {A hierarchical
  multilevel Markov chain Monte Carlo algorithm with applications to
  uncertainty quantification in subsurface flow}}, arXiv preprint
  arXiv:1303.7343,  (2013).

\bibitem{KlPl92}
{\sc P.E. Kloeden and E.~Platen}, {\em Numerical solution of stochastic
  differential equations}, vol.~23 of Applications of Mathematics (New York),
  Springer-Verlag, Berlin, 1992.

\bibitem{lawstuartDA2015}
{\sc Kody Law, Andrew Stuart, and Kostas Zygalakis}, {\em Data Assimilation: A
  Mathematical Introduction}, Springer-Verlag Berlin, 2015.

\bibitem{law2014deterministic}
{\sc Kody~JH Law, Hamidou Tembine, and Raul Tempone}, {\em Deterministic
  mean-field ensemble {K}alman filtering}, arXiv preprint arXiv:1409.0628v4,
  (2014).

\bibitem{le2011large}
{\sc Fran{\c{c}}ois Le~Gland, Val{\'e}rie Monbet, Vu-Duc Tran, et~al.}, {\em
  Large sample asymptotics for the ensemble {K}alman filter}, The Oxford
  Handbook of Nonlinear Filtering,  (2011), pp.~598--631.

\bibitem{ledoit2004well}
{\sc Olivier Ledoit and Michael Wolf}, {\em A well-conditioned estimator for
  large-dimensional covariance matrices}, Journal of multivariate analysis, 88
  (2004), pp.~365--411.

\bibitem{luenberger1968optimization}
{\sc David~G Luenberger}, {\em Optimization by vector space methods}, John
  Wiley \& Sons, 1968.

\bibitem{mandel2011convergence}
{\sc Jan Mandel, Loren Cobb, and Jonathan~D Beezley}, {\em On the convergence
  of the ensemble {K}alman filter}, Applications of Mathematics, 56 (2011),
  pp.~533--541.

\bibitem{Mishra12}
{\sc S.~Mishra and C.~Schwab}, {\em Sparse tensor multi-level {M}onte {C}arlo
  finite volume methods for hyperbolic conservation laws with random initial
  data}, Math. Comp., 81 (2012), pp.~1979--2018.

\bibitem{pajonk2012deterministic}
{\sc Oliver Pajonk, Bojana~V Rosi{\'c}, Alexander Litvinenko, and Hermann~G
  Matthies}, {\em {A deterministic filter for non-Gaussian Bayesian
  estimationÑApplications to dynamical system estimation with noisy
  measurements}}, Physica D: Nonlinear Phenomena, 241 (2012), pp.~775--788.

\bibitem{Ren01}
{\sc Yao-Feng Ren and Han-Ying Liang}, {\em On the best constant in
  {M}arcinkiewicz-{Z}ygmund inequality}, Statist. Probab. Lett., 53 (2001),
  pp.~227--233.

\end{thebibliography}
\bibliographystyle{siam}

\end{document}